\documentclass[12pt,a4paper]{article}
\usepackage{amsmath,amssymb,amsthm}

\def\N{\mathbb N}
\def\R{\mathbb R}
\def\oR{\overline{\mathbb R}}

\def\mC{\mathcal{C}}
\def\mG{\mathcal{G}}

\def\mJ{\mathcal{J}}

\def\mQ{\mathcal{Q}}
\def\mU{\mathcal{U}}

\def\be{{\bar{e}}}\def\we{\widetilde{e}}\def\he{\widehat{e}}

\def\op{\bar{p}}

\def\ou{\bar{u}}\def\wu{\widetilde{u}}\def\hu{\widehat{u}}
\def\ov{\bar{v}}

\def\ox{\bar{x}}

\def\oB{\bar{B}}

\def\gph{{\rm gph\,}}

\def\dom{{\rm dom\,}}

\def\conv{{\rm conv}}

\def\proj{{\rm proj}}
\def\inter{{\rm int}}

\def\st{|\,}\def\bst{\big|\,}\def\Bst{\Big|\,}
\def\Limsup{\mathop{{\rm Lim\hspace{0.3mm}sup}}}

\def\limsup{\mathop{{\rm lim\hspace{0.3mm}sup}}}
\def\liminf{\mathop{{\rm lim\hspace{0.3mm}inf}}}

\let\oldae\ae
\renewcommand{\ae}{\text{\rm\oldae}}

\begin{document}

\newtheorem{Theorem}{Theorem}[section]
\newtheorem{Proposition}[Theorem]{Proposition}
\newtheorem{Remark}[Theorem]{Remark}
\newtheorem{Lemma}[Theorem]{Lemma}
\newtheorem{Corollary}[Theorem]{Corollary}
\newtheorem{Definition}[Theorem]{Definition}
\newtheorem{Example}[Theorem]{Example}
\newtheorem{CounterExample}[Theorem]{Counter-Example}
\renewcommand{\theequation}{\thesection.\arabic{equation}}
\normalsize

\title{Subgradients of Marginal Functions in Parametric Control Problems of
       Partial Differential Equations}

\author{Nguyen Thanh Qui\footnote{College of Information and Communication Technology,
        Can Tho University, Campus II, 3/2 Street, Can Tho, Vietnam;
        ntqui@cit.ctu.edu.vn. Institut f\"{u}r Mathematik,
        Universit\"{a}t W\"{u}rzburg, Emil-Fischer-Str.~30, 97074 W\"{u}rzburg, Germany;
        thanhqui.nguyen@mathematik.uni-wuerzburg.de. The research of this author was supported
        by the Alexander von Humboldt Foundation.}\quad and\quad
        Daniel Wachsmuth\footnote{Institut f\"{u}r Mathematik, Universit\"{a}t W\"{u}rzburg,
        Emil-Fischer-Str.~30, 97074 W\"{u}rzburg, Germany;
        daniel.wachsmuth@mathematik.uni-wuerzburg.de.}}

\maketitle
\date{}

\noindent {\bf Abstract.} The paper studies generalized
differentiability properties of the marginal function of parametric
optimal control problems of semilinear elliptic partial differential
equations. We establish upper estimates for the regular and the
limiting subgradients of the marginal function. With some additional
assumptions, we show that the solution map of the perturbed optimal
control problems has local upper H\"{o}lderian selections. This
leads to a lower estimate for the regular subdifferential of the
marginal function.



\section{Introduction}

It is well recognized that \emph{optimal value function} (or,
\emph{marginal function}) and \emph{solution map} of parametric
optimization problems are very important in variational analysis,
optimization theory, control theory, etc. Problems of investigation
on generalized differentiability properties of the marginal function
and the solution map of parametric optimization problems are in the
research direction of differential stability of optimization
problems. Many researchers have had contributions to this research
direction such as Aubin \cite{Aub98B}, Auslender \cite{Aus79MP},
Bonnans and Shapiro \cite{BonSha00B}, Dien and Yen
\cite{DieYen91AMO}, Gauvin and Dubeau \cite{GauDub82MP,GauDub84MP},
Gollan \cite{Gol84MOR}, Mordukhovich et al.
\cite{MorNm05MOR,MorNmYn09MP}, Rockafellar \cite{Rock82MP}, Thibault
\cite{Thi91SICON}. In general, marginal functions are complicated
and intrinsically nonsmooth in perturbed parameters, therefore
generalized differentiability properties of marginal functions play
a crucial role in order to derive important information on
sensitivity and stability of optimization problems.

Recently, Mordukhovich et al. \cite{MorNmYn09MP} derived formulas
for computing and estimating the regular subdifferential and the
limiting (Mordukhovich, singular) subdifferentials of marginal
functions in Banach spaces and specified these results for important
classes of problems in parametric optimization with smooth and
nonsmooth data. Motivated by the results of \cite{MorNmYn09MP}, some
new results on differential stability of convex optimization
problems under inclusion constraints as well as under Aubin's
regularity condition have been provided in
\cite{AnYao16JOTA,AnYen15AA}. In addition, differential stability of
parametric optimal control problems governed by ordinary
differential equations (ODEs) has been studied by many authors in
\cite{AnToa18AMV,ChKiTo16JOTA,Toan15TJM,ToaKie10SVVA,ThuToa16JOTA},
where many results on the first-order behavior of the marginal
function of parametric continuous/discrete optimal control problems
with linear constraints, convex/nonconvex cost functionals have been
established.

To the best of our knowledge, although there were many works on
differential stability of optimal control problems of ODEs, the
problem of study on differential stability of optimal control
problems governed by partial differential equations (PDEs) remains
open. For this reason, in the present paper we focus on the study of
generalized differentiability properties of the marginal function of
perturbed optimal control problems for PDEs. Namely, we will
establish new formulas for computing/estimating the regular
subdifferential as well as the limiting subdifferentials in the
Mordukhovich's sense of the marginal function of perturbed optimal
control problems of semilinear elliptic PDEs with control
constraints.

For the original control problem in question, we are interested in
two classes of perturbed control problems with respect to two
different parametric spaces. In the first class of perturbed
problems, under some standard assumptions posed on the initial data
of the original control problem, we establish new upper estimates
for the regular subdifferential, the Mordukhovich subdifferential,
and the singular subdifferential of the marginal function of the
perturbed control problems. In addition, these upper estimates for
the regular and the Mordukhovich subdifferentials of the marginal
function will hold as equalities provided that the solution map of
the perturbed control problems has a local upper Lipschitzian
selection at the reference point in the graph of the solution map.
For the second class of perturbed problems, we consider parametric
bang-bang control problems, where the cost functional of such
control problems does not involve the usual quadratic term for the
control. With some additional assumptions to the above standard
assumptions, we show that the solution map of the perturbed control
problems admits a local upper H\"{o}lderian selection at the
reference point in the graph of the solution map. This leads to a
new lower estimate for the regular subdifferential of the marginal
function.

The rest of the paper is organized as follows. A class of optimal
control problems together with standard assumptions in optimal
control theory of PDEs and auxiliary results are stated in
Section~2. In Section 3, we establish upper estimates for the
regular, the Mordukhovich, and the singular subdifferentials of the
marginal function to the first class of perturbed control problems.
Section~4 is devoted to prove the existence of local upper
H\"{o}lderian selections of the solution map and new lower estimates
for the regular subdifferential of the marginal function to the
second class of perturbed control problems (parametric bang-bang
control problems). Some concluding remarks and open problems are
provided in the last section.

\section{Preliminaries}
\setcounter{equation}{0}

\subsection{Control problem statement}

The original optimal control problem that we are interested in this
paper is stated as follows
\begin{equation}\label{OptConPro}
\begin{cases}
    {\rm Minimize}   &J(u)=\displaystyle\int_\Omega L\big(x,y_u(x)\big)dx
                          +\frac{1}{2}\displaystyle\int_\Omega\zeta(x)u(x)^2dx\\
    {\rm subject~to} &\alpha(x)\leq u(x)\leq\beta(x)\quad\mbox{for a.a.}\ x\in\Omega,\\
\end{cases}
\end{equation}
where $\zeta\in L^2(\Omega)$ with $\zeta(x)\geq0$ for a.a.
$x\in\Omega$, and $y_u$ is the weak solution of the following
Dirichlet problem
\begin{equation}\label{StateEq}
\begin{cases}
\begin{aligned}
    Ay+f(x,y)&=u\ &&\mbox{in}\ \Omega\\
            y&=0  &&\mbox{on}\ \Gamma,
\end{aligned}
\end{cases}
\end{equation}
where the letter $A$ denotes the second-order elliptic differential
operator of the form
$$Ay(x)=-\sum_{i,j=1}^N\partial_{x_j}\big(a_{ij}(x)\partial_{x_i}y(x)\big).$$
The corresponding perturbed control problem of \eqref{OptConPro} is
given below
\begin{equation}\label{PerProWtCd}
\begin{cases}
    {\rm Minimize}   & \mJ(u,e)=J(u+e_y)+(e_J,y_{u+e_y})_{L^2(\Omega)}\\
    {\rm subject~to} & u\in\mG(e)=\mU_{ad}(e)\cap\mQ,
\end{cases}
\end{equation}
where $J(\cdot)$ is the cost functional of
problem~\eqref{OptConPro}, $y_{u+e_y}$ is the weak solution of the
perturbed Dirichlet problem
\begin{equation}\label{PerStaEqWtCd}
\begin{cases}
\begin{aligned}
    Ay+f(x,y)&=u+e_y\ &&\mbox{in}\ \Omega\\
            y&=0      &&\mbox{on}\ \Gamma,
\end{aligned}
\end{cases}
\end{equation}
$\mQ$ is a given subset of $L^{p_0}(\Omega)$, and
\begin{equation}\label{AdCtrStUade}
    \mU_{ad}(e)=\big\{u\in L^{q_0}(\Omega)\bst
    (\alpha+e_\alpha)(x)\leq u(x)\leq(\beta+e_\beta)(x)~\mbox{for a.a.}~x\in\Omega\big\}.
\end{equation}
We introduce $E=L^{p_1}(\Omega)\times L^{p_2}(\Omega)\times
L^{p_3}(\Omega)\times L^{p_4}(\Omega)$ the parametric space with the
norm of $e=(e_J,e_y,e_\alpha,e_\beta)\in E$ given by
\begin{equation}\label{DefNormE}
    \|e\|_E=\|e_J\|_{L^{p_1}(\Omega)}+\|e_y\|_{L^{p_2}(\Omega)}+\|e_\alpha\|_{L^{p_3}(\Omega)}
           +\|e_\beta\|_{L^{p_4}(\Omega)}.
\end{equation}
In what follows, we will write $\mU_{ad}$ for $\mU_{ad}(0)$ the set
of \emph{admissible controls} of problem~\eqref{OptConPro}.

Let us recall definitions of the marginal function and the solution
map of the perturbed control problem \eqref{PerProWtCd}. The
\emph{marginal function} $\mu:E\to\oR$ of the perturbed
problem~\eqref{PerProWtCd} is defined by
\begin{equation}\label{MrgnlFunc}
    \mu(e)=\inf_{u\in\mG(e)}\mJ(u,e),
\end{equation}
and the \emph{solution}/\emph{argminimum map} $S:E\rightrightarrows
L^{s_0}(\Omega)$ of problem~\eqref{PerProWtCd} is given by
\begin{equation}\label{SolMap}
    S(e)=\big\{u\in\mG(e)\bst\mu(e)=\mJ(u,e)\big\}.
\end{equation}
The main goal of this paper is to establish explicit formulas for
computing/estimating the regular subdifferential, the Mordukhovich
subdifferential, and the singular subdifferential of the marginal
function $\mu(\cdot)$ in \eqref{MrgnlFunc} at a given parameter
$\be\in E$.

\subsection{Generalized differentiation from variational analysis}

Let us recall some material on generalized differentiation taken
from \cite{Mor06Ba}. Unless otherwise stated, every reference norm
in a product normed space is the sum norm. Given a point $u$ in a
Banach space $X$ and $\rho>0$, we denote $B_\rho(u)$ the open ball
of center $u$ and radius $\rho$ in $X$, and $\oB_\rho(u)$ is the
corresponding closed ball. In particular, for any $p\in[1,\infty]$,
the notation $\oB^p_\rho(u)$ stands for the closed ball
$\oB_\rho(u)$ in the space $L^p(\Omega)$, i.e.,
$$\oB^p_\rho(u)=\big\{v\in L^p(\Omega)\bst\|v-u\|_{L^p(\Omega)}\leq\rho\big\}.$$
Let $F:X\rightrightarrows W$ be a multifunction between Banach
spaces. The \emph{graph} and the \emph{domain} of $F$ are the sets
$\gph F:=\{(u,v)\in X\times W\st v\in F(u)\}$ and $\dom F:=\{u\in
X\st F(u)\neq\emptyset\}$, respectively. We say that $F$ is locally
closed around the point $\bar\omega=(\ou,\ov)\in\gph F$ if $\gph F$
is locally closed around $\bar\omega$, i.e., there exists a closed
ball $\oB_\rho(\bar\omega)$ such that $\oB_\rho(\bar\omega)\cap\gph
F$ is closed in $X\times W$.

For a multifunction $\Phi:X\rightrightarrows X^*$, the
\emph{sequential Painlev\'e-Kuratowski upper limit} of $\Phi$ as
$u\to \ou$ is defined by
\begin{equation}\label{OtrLimit}
\begin{aligned}
    \displaystyle\Limsup_{u\to\ou}\Phi(u)=\Big\{
    &u^*\in X^*\Bst\mbox{there exist}~u_n\to\ou~\mbox{and}~u^*_n\stackrel{w^*}\rightharpoonup u^*~\mbox{with}\\
    &u^*_n\in\Phi(u_n)~\mbox{for every}~k\in\N=\{1,2,\dots\}\Big\}.
\end{aligned}
\end{equation}
Given an extended-real-valued function $\phi:X\to\oR$ and
$\ou\in\dom\phi:=\{u\in X\st\phi(u)<\infty\}$, the \emph{regular
subdifferential} (also called the \emph{Fr\'echet subdifferential})
of $\phi$ at the point $\ou$ is the set
$\widehat{\partial}\phi(\ou):=\widehat{\partial}_0\phi(\ou)$, where
$\widehat{\partial}_\varepsilon\phi(\ou)$ with $\varepsilon\geq0$ is
the collection of $\varepsilon$-\emph{subgradients} of $\phi$ at
$\ou$ defined by
\begin{equation}\label{RegSubdif}
    \widehat{\partial}_\varepsilon\phi(\ou)=\bigg\{u^*\in X^*\Bst\liminf_{u\to\ou}
    \frac{\phi(u)-\phi(\ou)-\langle u^*,u-\ou\rangle}{\|u-\ou\|}\geq-\varepsilon\bigg\},
\end{equation}
and the \emph{regular}/\emph{Fr\'echet upper subdifferential} of
$\phi$ at $\ou$ is given by
\begin{equation}\label{RegUpSbdif}
    \widehat{\partial}^+\phi(\ou)=-\widehat{\partial}(-\phi)(\ou).
\end{equation}
The \emph{limiting basic subdifferential} (the \emph{Mordukhovich
subdifferential}) of $\phi$ at $\ou$ is defined via the sequential
outer limit \eqref{OtrLimit} by
\begin{equation}\label{LimSubdif}
    \partial\phi(\ou)=\Limsup_{\substack{u\stackrel{\phi}\longrightarrow\ou \\ \varepsilon\downarrow0}}
    \widehat{\partial}_\varepsilon\phi(u),
\end{equation}
and the \emph{limiting singular subdifferential} (the \emph{singular
subdifferential} for short) of $\phi$ at $\ou$ is given by
\begin{equation}\label{SngSubdif}
    \partial^\infty\phi(\ou)
    =\Limsup_{\substack{ u\stackrel{\phi}\longrightarrow\ou\\ \varepsilon,\lambda\downarrow0}}
    \lambda\widehat{\partial}_\varepsilon\phi(u),
\end{equation}
where the notation $u\stackrel{\phi}\to\ou$ means that $u\to\ou$
with $\phi(u)\to\phi(\ou)$.

Note that we can equivalently put $\varepsilon=0$ in
\eqref{LimSubdif} and \eqref{SngSubdif} if $X$ is an \emph{Asplund
space} \cite{Asp68AM} (see also \cite{Mor06Ba,Phelp93B} for more
details) and $\phi$ is lower semicontinuous around $\ou$. It is
obvious that $\widehat{\partial}\phi(\ou)\subset\partial\phi(\ou)$
whenever $\phi(\ou)$ is finite. If the latter inclusion holds as
equality, $\phi$ is said to be \emph{lower regular} at $\ou$. The
class of lower regular functions is sufficiently large and important
in variational analysis and optimization; see
\cite{Mor06Ba,Mor06Bb,RoWe98B} for more details and applications.

Given a nonempty set $\Theta\subset X$, the \emph{regular and
Mordukhovich normal cones} to $\Theta$ at $\ou\in\Theta$ are
respectively defined by
\begin{equation}\label{RgLmtgNrCn}
    \widehat{N}(\ou;\Theta)=\widehat{\partial}\delta(\ou;\Theta)
    \quad\mbox{and}\quad N(\ou;\Theta)=\partial\delta(\ou;\Theta),
\end{equation}
where $\delta(\cdot;\Theta)$ is the indicator function of $\Theta$
defined by $\delta(u;\Theta)=0$ for $u\in\Theta$ and
$\delta(u;\Theta)=\infty$ otherwise. If $X$ is an Asplund space and
$\Theta\subset X$ is locally closed around $\ou\in\Omega$, we have
\begin{equation}\label{LmNrCnFrCn}
    N(\ou;\Theta)=\Limsup_{u\stackrel{\Theta}\to\ou}\widehat{N}(u;\Theta).
\end{equation}
The \emph{regular} and
\emph{Mordukhovich coderivatives} of the multifunction
$F:X\rightrightarrows W$ at the point $(\ou,\ov)\in\gph F$ are
respectively the multifunction
$\widehat{D}^*F(\ou,\ov):W^*\rightrightarrows X^*$ defined by
\begin{equation}\label{DefFrDrvtv}
    \widehat{D}^*F(\ou,\ov)(v^*)=\big\{u^*\in X^*\bst(u^*,-v^*)
    \in\widehat{N}\big((\ou,\ov);\gph F\big)\big\},~\forall v^*\in W^*,
\end{equation}
and the multifunction $D^*F(\ou,\ov):W^*\rightrightarrows X^*$ given
by
\begin{equation}\label{DefMrDrvtv}
    D^*F(\ou,\ov)(v^*)=\big\{u^*\in X^*\bst(u^*,-v^*)\in N\big((\ox,\ov);\gph F\big)\big\},~\forall v^*\in W^*.
\end{equation}
The multifunction $F$ is said to be \emph{normally regular} at
$(\ou,\ov)$ if
$$\widehat{D}^*F(\ou,\ov)(v^*)=D^*F(\ou,\ov)(v^*),~\forall v^*\in W^*.$$
If $F=f:X\to W$ is a single-valued function, we will write
$\widehat{D}^*f(\ou)(v^*)$ and $D^*f(\ou)(v^*)$ for coderivatives of
$f$ in \eqref{DefFrDrvtv} and \eqref{DefMrDrvtv}. If $f$ is
respectively Fr\'echet differentiable and strictly differentiable at
$\ou$, both the regular and Mordukhovich coderivatives are
extensions of the corresponding \emph{adjoint derivative operators}
in the sense that
$$\widehat{D}^*f(\ou)(v^*)=f'(\ou)^*v^*\quad\mbox{and}\quad D^*f(\ou)(v^*)=f'(\ou)^*v^*,~\forall v^*\in W^*.$$

The multifunction $F:X\rightrightarrows W$ is \emph{locally
Lipschitz-like} (or, $F$ has the \emph{Aubin property}
\cite{DontRock09B}) around a point $(\ou,\ov)\in\gph F$ if there
exist $\ell>0$ and neighborhoods $U$ of $\ou$, $V$ of $\ov$ such
that
$$F(u_1)\cap V\subset F(u_2)+\ell\|u_1-u_2\|\oB_W,~\forall u_1,u_2\in U,$$
where $\oB_W$ denotes the closed unit ball in $W$. Characterization
of this property via the mixed Mordukhovich coderivative of $F$ can
be found in \cite[Theorem~4.10]{Mor06Ba}. Following Robinson
\cite{Robin79MP}, a single-valued function $h:D\subset X\to W$ is
\emph{locally upper Lipschitzian} at $\ou$ if there exist $\eta>0$
and $\ell>0$ such that
\begin{equation}\label{HLipSlect}
    \|h(u)-h(\ou)\|\leq\ell\|u-\ou\|\quad\mbox{whenever}\quad u\in B_\eta(\ou)\cap D.
\end{equation}
We say that a multifunction $F:D\rightrightarrows W$ defined on some
set $D\subset X$ admits a \emph{local upper Lipschitzian selection}
at $(\ou,\ov)\in\gph F$ if there is a single-valued function $h:D\to
W$, which is locally upper Lipschitzian at $\ou$ satisfying
$h(\ou)=\ov$ and $h(u)\in F(u)$ for all $u\in D$ in a neighborhood
of $\ou$. We also call $h$ a \emph{local upper H\"{o}lderian
selection} at $(\ou,\ov)\in\gph F$ if \eqref{HLipSlect} is replaced
by the H\"{o}lder property with some exponent $\ae\geq0$ below
\begin{equation}\label{HHldrSlect}
    \|h(u)-h(\ou)\|\leq\ell\|u-\ou\|^{\ae}\quad\mbox{whenever}\quad u\in B_\eta(\ou)\cap D.
\end{equation}

\subsection{Assumptions and auxiliary results}

Let us assume that $\Omega\subset\R^N$ with $N\in\{1,2,3\}$,
$\alpha,\beta\in L^\infty(\Omega)$, $\alpha\leq\beta$, and
$\alpha\not\equiv\beta$. Moreover, $L,f:\Omega\times\R\to\R$ are
Carath\'eodory functions of class $\mC^2$ with respect to the second
variable satisfying the following assumptions.

\textbf{(A1)} The function $f(\cdot,0)\in L^{\op}(\Omega)$ with
$\op>N/2$,
$$\dfrac{\partial f}{\partial y}(x,y)\geq0\quad\mbox{for a.a.}\ x\in\Omega,$$
and for all $M>0$ there exists a constant $C_{f,M}>0$ such that
$$\left|\dfrac{\partial f}{\partial y}(x,y)\right|+\left|\dfrac{\partial^2f}{\partial y^2}(x,y)\right|\leq C_{f,M}
  \quad\mbox{for a.a.}\ x\in\Omega\ \mbox{and}\ |y|\leq M.$$
For every $M>0$ and $\varepsilon>0$ there exists $\delta>0$,
depending on $M$ and $\varepsilon$ such that
$$\left|\dfrac{\partial^2f}{\partial y^2}(x,y_2)-\dfrac{\partial^2f}{\partial y^2}(x,y_1)\right|<\varepsilon
  \quad\mbox{if}\ |y_1|,|y_2|\leq M,|y_2-y_1|\leq\delta,\ \mbox{and for a.a.}\ x\in\Omega.$$

\textbf{(A2)} The function $L(\cdot,0)\in L^1(\Omega)$ and for all
$M>0$ there are a constant $C_{L,M}>0$ and a function $\psi_M\in
L^{\op}(\Omega)$ such that for every $|y|\leq M$ and almost all
$x\in\Omega$,
$$\left|\dfrac{\partial L}{\partial y}(x,y)\right|\leq\psi_M(x),\quad
  \left|\dfrac{\partial^2L}{\partial y^2}(x,y)\right|\leq C_{L,M}.$$
For every $M>0$ and $\varepsilon>0$ there exists $\delta>0$,
depending on $M$ and $\varepsilon$ such that
$$\left|\dfrac{\partial^2L}{\partial y^2}(x,y_2)-\dfrac{\partial^2L}{\partial y^2}(x,y_1)\right|<\varepsilon
  \quad\mbox{if}\ |y_1|,|y_2|\leq M,|y_2-y_1|\leq\delta,\ \mbox{and for a.a.}\ x\in\Omega.$$

\textbf{(A3)} The set $\Omega$ is an open and bounded domain in
$\R^N$ with Lipschitz boundary $\Gamma$. The set $\mQ$ is closed,
convex, and bounded in $L^{p_0}(\Omega)$ satisfying
$\mU_{ad}(e)\cap\inter\mQ\neq\emptyset$ for some $e\in E$, where
$\inter\mQ$ stands for the interior of $\mQ$. The coefficients
$a_{ij}\in C(\bar\Omega)$ of the second-order elliptic differential
operator $A$ satisfy
$$\lambda_A\|\xi\|^2_{\R^N}\leq\sum_{i,j=1}^Na_{ij}(x)\xi_i\xi_j,\ \forall\xi\in\R^N,\ \mbox{for a.a.}\ x\in\Omega,$$
for some constant $\lambda_A>0$.

For every $u\in L^p(\Omega)$ with $p>N/2$, according to
\cite[Chapter~4]{Trolt10B}, equation \eqref{StateEq} has a unique
weak solution $y_u\in H^1_0(\Omega)\cap C(\bar\Omega)$. In addition,
there exists a constant $M_{\alpha,\beta}>0$ such that
\begin{equation}\label{EstSolEqSt}
    \|y_u\|_{H^1_0(\Omega)}+\|y_u\|_{C(\bar\Omega)}\leq M_{\alpha,\beta},~\forall u\in\mU_{ad}.
\end{equation}
The \emph{control-to-state mapping} $G:L^p(\Omega)\to
H^1_0(\Omega)\cap C(\bar\Omega)$ defined by $G(u)=y_u$ is of class
$\mC^2$. Moreover, for every $v\in L^2(\Omega)$, $z_{u,v}=G'(u)v$ is
the unique weak solution of
\begin{equation}\label{EqSolZuv}
\begin{cases}
  \begin{aligned}
     Az+\frac{\partial f}{\partial y}(x,y)z&=v\ &&\mbox{in}\ \Omega\\
                                          z&=0  &&\mbox{on}\ \Gamma,
  \end{aligned}
\end{cases}
\end{equation}
and for any $v_1,v_2\in L^2(\Omega)$, $w_{v_1,v_2}=G''(u)(v_1,v_2)$
is the unique weak solution of
\begin{equation}\label{EqSolSeGvv}
\begin{cases}
  \begin{aligned}
     Aw+\frac{\partial f}{\partial y}(x,y)w+\frac{\partial^2f}{\partial y^2}(x,y)z_{u,v_1}z_{u,v_2}
      &=0\ &&\mbox{in}\ \Omega\\
     w&=0  &&\mbox{on}\ \Gamma,
  \end{aligned}
\end{cases}
\end{equation}
where $y=G(u)$ and $z_{u,v_i}=G'(u)v_i$ for $i=1,2$.

By assumption \textbf{(A2)}, using the latter results and applying
the chain rule we deduce that the cost functional
$J:L^p(\Omega)\to\R$ with $p>N/2$ is of class $\mC^2$, and the first
and second derivatives of $J(\cdot)$ are given by
\begin{equation}\label{FDeriCost}
    J'(u)v=\int_\Omega(\varphi_u+\zeta u)vdx,
\end{equation}
and
\begin{equation}\label{SDeriCost}
    J''(u)(v_1,v_2)=\int_\Omega\left(\dfrac{\partial^2L}{\partial y^2}(x,y_u)z_{u,v_1}z_{u,v_2}
    +\zeta v_1v_2-\varphi_u\dfrac{\partial^2f}{\partial y^2}(x,y_u)z_{u,v_1}z_{u,v_2}\right)dx,
\end{equation}
where $z_{u,v_i}=G'(u)v_i$ for $i=1,2$, and $\varphi_u\in
H^1_0(\Omega)\cap C(\bar\Omega)$ is the adjoint state of $y_u$
defined as the unique  weak solution of
\begin{equation}\label{AdjStaEq}
\begin{cases}
\begin{aligned}
    A^*\varphi+\dfrac{\partial f}{\partial y}(x,y_u)\varphi
                     &=\dfrac{\partial L}{\partial y}(x,y_u)\ &&\mbox{in}\ \Omega\\
             \varphi &=0                                      &&\mbox{on}\ \Gamma,
\end{aligned}
\end{cases}
\end{equation}
where $A^*$ is the adjoint operator of $A$.

A control $\ou\in\mU_{ad}$ is said to be a
\emph{solution}/\emph{global minimum} of problem~\eqref{OptConPro}
if $J(\ou)\leq J(u)$ for all $u\in\mU_{ad}$. We will say that $\ou$
is a \emph{local solution}/\emph{local minimum} of
problem~\eqref{OptConPro} in the sense of $L^p(\Omega)$ if there
exists a closed ball $\oB^p_\varepsilon(\ou)$ such that $J(\ou)\leq
J(u)$ for all $u\in\mU_{ad}\cap\oB^p_\varepsilon(\ou)$. The local
solution $\ou$ is called \emph{strict} if $J(\ou)<J(u)$ holds for
all $u\in\mU_{ad}\cap\oB^p_\varepsilon(\ou)$ with $u\neq\ou$. Under
the assumptions given above, solutions of problem~\eqref{OptConPro}
exist. We introduce the space $Y=H^1_0(\Omega)\cap C(\bar\Omega)$
endowed with the norm
$$\|y\|_Y=\|y\|_{H^1_0(\Omega)}+\|y\|_{L^\infty(\Omega)}.$$
According to \cite[Chapter~4]{Trolt10B}, if $\ou\in\mU_{ad}$ is a
solution of problem~\eqref{OptConPro} in the sense of $L^p(\Omega)$,
then there exist a unique state $y_{\ou}\in Y$ and a unique adjoint
state $\varphi_{\ou}\in Y$ satisfying the first-order optimality
system
\begin{equation}\label{StateEqSol}
\begin{cases}
\begin{aligned}
    Ay_{\ou}+f(x,y_{\ou})&=\ou\ &&\mbox{in}\ \Omega\\
                  y_{\ou}&=0    &&\mbox{on}\ \Gamma,
\end{aligned}
\end{cases}
\end{equation}
\begin{equation}\label{AdjEq}
\begin{cases}
\begin{aligned}
    A^*\varphi_{\ou}+\dfrac{\partial f}{\partial y}(x,y_{\ou})\varphi_{\ou}
                 &=\dfrac{\partial L}{\partial y}(x,y_{\ou})\ &&\mbox{in}\ \Omega\\
    \varphi_{\ou}&=0                                          &&\mbox{on}\ \Gamma,
\end{aligned}
\end{cases}
\end{equation}
\begin{equation}\label{VarIneq}
    \int_\Omega(\varphi_{\ou}+\zeta\ou)(u-\ou)dx\geq0,~\forall u\in\mU_{ad}.
\end{equation}
Similarly, if $\ou_e\in\mG(e)$ is a solution of the perturbed
problem \eqref{PerProWtCd} with respect to $e\in E$, then $\ou_e$
satisfies the perturbed first-order optimality system
\begin{equation}\label{PertbStateEq}
\begin{cases}
\begin{aligned}
    Ay_{\ou_e+e_y}+f(x,y_{\ou_e+e_y})&=\ou_e+e_y\ &&\mbox{in}\ \Omega\\
                        y_{\ou_e+e_y}&=0        &&\mbox{on}\ \Gamma,
\end{aligned}
\end{cases}
\end{equation}
\begin{equation}\label{PertbAdjEq}
\begin{cases}
\begin{aligned}
    A^*\varphi_{\ou_e,e}+\dfrac{\partial f}{\partial y}(x,y_{\ou_e+e_y})\varphi_{\ou_e,e}
                     &=\dfrac{\partial L}{\partial y}(x,y_{\ou_e+e_y})+e_J\ &&\mbox{in}\ \Omega\\
    \varphi_{\ou_e,e}&=0                                                    &&\mbox{on}\ \Gamma,
\end{aligned}
\end{cases}
\end{equation}
\begin{equation}\label{PertbVarIneq}
    \int_\Omega(\varphi_{\ou_e,e}+\zeta\ou_e)\big(u(x)-\ou_e(x)\big)dx\geq0,\ \forall u\in\mG(e),
\end{equation}
where $\varphi_{\ou_e,e}$ is the adjoint state of $y_{\ou_e+e_y}$
for the perturbed problem~\eqref{PerProWtCd}. Furthermore, the
partial derivative of $\mJ(u,e)$ in $u$ at $\ou_e$ can be computed
by
\begin{equation}\label{PerbDermJ}
    \mJ'_u(\ou_e,e)v=\int_\Omega(\varphi_{\ou_e,e}+\zeta\ou_e)vdx.
\end{equation}

\section{Subgradients of marginal functions}
\setcounter{equation}{0}

In this section, we consider the parametric control
problem~\eqref{PerProWtCd}, where $p_0=2$ while $q_0=2$ in
\eqref{AdCtrStUade}, $p_1=p_2=p_3=p_4=2$ in \eqref{DefNormE}, and
$s_0=2$. This means that $\mQ\subset L^2(\Omega)$ and the perturbed
admissible control set $\mU_{ad}(e)\subset L^2(\Omega)$ for $e\in
E=L^2(\Omega)\times L^2(\Omega)\times L^2(\Omega)\times
L^2(\Omega)$.

\begin{Theorem}\label{ThmAxExSol}
Assume that the assumptions {\rm\textbf{(A1)}-\textbf{(A3)}} hold.
Then, for each $e\in E$ with $\mG(e)\neq\emptyset$, the perturbed
control problem~\eqref{PerProWtCd} has at least one optimal control
$\ou_e$ with associated optimal perturbed state $y_{\ou_e+e_y}\in
H^1(\Omega)\cap C(\bar\Omega)$.
\end{Theorem}
\begin{proof}
Let $e\in E$ be such that $\mG(e)\neq\emptyset$. Then, $\mG(e)$ is
nonempty, closed, bounded, and convex in $L^2(\Omega)$ due to
$\mU_{ad}(e)$ is closed, bounded, and convex in $L^2(\Omega)$. By
arguing similarly as in the proof of \cite[Theorem~4.1]{QuiWch17},
we obtain assertion of the theorem. $\hfill\Box$
\end{proof}

\subsection{Regular subgradients of marginal functions}

Let the marginal function $\mu(\cdot)$ from \eqref{MrgnlFunc} be
finite at some $\be\in\dom S$, and let $\ou_{\be}\in S(\be)$ be such
that $\widehat{\partial}^+\mJ(\ou_{\be},\be)\neq\emptyset$. Then,
applying \cite[Theorem~1]{MorNmYn09MP}, we obtain
\begin{equation}\label{FreSubEst}
    \widehat{\partial}\mu(\be)\subset
    \bigcap_{(u^*,e^*)\in\widehat{\partial}^+\mJ(\ou_{\be},\be)}
    \Big(e^*+\widehat{D}^*\mG(\be,\ou_{\be})(u^*)\Big).
\end{equation}
Note that by the assumptions {\rm\textbf{(A1)}-\textbf{(A3)}} the
function $\mJ$ is Fr\'echet differentiable at $(\ou_{\be},\be)$.
Thus, we get
$$\widehat{\partial}^+\mJ(\ou_{\be},\be)=\{\nabla\mJ(\ou_{\be},\be)\}
  =\big\{\big(\mJ'_u(\ou_{\be},\be),\mJ'_e(\ou_{\be},\be)\big)\big\}.$$
Consequently, from \eqref{FreSubEst} we deduce that
\begin{equation}\label{FrSbEstNoCp}
    \widehat{\partial}\mu(\be)
    \subset\mJ'_e(\ou_{\be},\be)+\widehat{D}^*\mG(\be,\ou_{\be})\big(\mJ'_u(\ou_{\be},\be)\big).
\end{equation}
If, in addition, the solution map $S:\dom\mG\rightrightarrows
L^2(\Omega)$ admits a local upper Lipschitzian selection at
$(\be,\ou_{\be})$, then by \cite[Theorem~2]{MorNmYn09MP} we obtain
\begin{equation}\label{FreSubEqul}
    \widehat{\partial}\mu(\be)
    =\mJ'_e(\ou_{\be},\be)+\widehat{D}^*\mG(\be,\ou_{\be})\big(\mJ'_u(\ou_{\be},\be)\big).
\end{equation}
We will apply \eqref{FrSbEstNoCp} to derive a new explicit formula
for estimating the Fr\'echet subdifferential
$\widehat{\partial}\mu(\be)$, and this formula will also be an exact
formula for computing $\widehat{\partial}\mu(\be)$ provided that the
solution map $S(\cdot)$ has a local upper Lipschitzian selection at
$(\be,\ou_{\be})$.

For each pair $(e,u)\in E\times L^2(\Omega)$ with $u\in\mG(e)$ we
define subsets $\Omega_1(e,u)$, $\Omega_2(e,u)$, $\Omega_3(e,u)$ of
$\Omega$ by
\begin{equation}\label{DjntSetsOm}
\begin{cases}
     \Omega_1(e,u)=\big\{x\in\Omega\st u(x)=\alpha(x)+e_\alpha(x)\big\},\\
     \Omega_2(e,u)=\big\{x\in\Omega\st u(x)\in\big(\alpha(x)+e_\alpha(x),\beta(x)+e_\beta(x)\big)\big\},\\
     \Omega_3(e,u)=\big\{x\in\Omega\st u(x)=\beta(x)+e_\beta(x)\big\}.
\end{cases}
\end{equation}
We have $\gph\mG=\gph\mU_{ad}\cap(E\times\mQ)$, where $\gph\mU_{ad}$
and $E\times\mQ$ are convex sets. In addition, we can verify that
$\gph\mU_{ad}\cap\inter(E\times\mQ)\neq\emptyset$. By
\cite[Proposition~1, p.~205]{IofTih79B}, we get
$$\begin{aligned}
    \widehat{N}\big((e,u);\gph\mG\big)
    &=\widehat{N}\big((e,u);\gph\mU_{ad}\big)+\widehat{N}\big((e,u);E\times\mQ\big)\\
    &=\widehat{N}\big((e,u);\gph\mU_{ad}\big)+\{0_E\}\times N(u;\mQ).
\end{aligned}$$
Thus, for each $(\be,\ou_{\be})\in\gph S$, we obtain
\begin{equation}\label{FreGSmRle}
\begin{aligned}
    \widehat{D}^*\mG(\be,\ou_{\be})(u^*)
    &=\big\{e^*\in E^*\bst (e^*,-u^*)\in\widehat{N}\big((\be,\ou_{\be});\gph\mG\big)\big\}\\
    &=\big\{e^*\in E^*\bst (e^*,-u^*)\in\widehat{N}\big((\be,\ou_{\be});\gph\mU_{ad}\big)+
     \{0_E\}\times N(\ou_{\be};\mQ)\big\}.
\end{aligned}
\end{equation}
In order to compute $\widehat{D}^*\mG(\be,\ou_{\be})(u^*)$
explicitly via \eqref{FreGSmRle}, we provide a formula for computing
the regular normal cone
$\widehat{N}\big((\be,\ou_{\be});\gph\mU_{ad}\big)$ in the following
lemma.

\begin{Lemma}\label{LmFNCGpUad}
Assume that the assumptions {\rm\textbf{(A1)}-\textbf{(A3)}} hold
and let $\ou_{\be}\in S(\be)$. The following formula holds that
\begin{equation}\label{FrNrCnGpUad}
\begin{aligned}
    \widehat{N}\big((\be,\ou_{\be});\gph\mU_{ad}\big)
    &=\Big\{(e^*,u^*)\in E^*\times L^2(\Omega)\Bst e^*=(0,0,e^*_\alpha,e^*_\beta),\,u^*=-e^*_\alpha-e^*_\beta,\\
    &\qquad e^*_\alpha|_{\Omega_1(\be,\ou_{\be})}\geq0,\,
            e^*_\alpha|_{\Omega\setminus\Omega_1(\be,\ou_{\be})}=0,\\
    &\qquad e^*_\beta|_{\Omega_3(\be,\ou_{\be})}\leq0,\,e^*_\beta|_{\Omega\setminus\Omega_3(\be,\ou_{\be})}=0\Big\}.
\end{aligned}
\end{equation}
\end{Lemma}
\begin{proof}
Let $(e^*,u^*)\in\widehat{N}\big((\be,\ou_{\be});\gph\mU_{ad}\big)$
with $e^*=(e^*_y,e^*_J,e^*_\alpha,e^*_\beta)\in E^*$. Since the set
$\mU_{ad}(e)$ does not depend on $e_y$ and $e_J$ for every
$e=(e_y,e_J,e_\alpha,e_\beta)\in E$, we have $e^*_y=e^*_J=0$. By the
definition of $\widehat{N}\big((\be,\ou_{\be});\gph\mU_{ad}\big)$,
we have
\begin{equation}\label{LmspFrCne}
    \limsup_{(e,u)\stackrel{\gph\mU_{ad}}\longrightarrow(\be,\ou_{\be})}
    \frac{\langle e^*_\alpha,e_\alpha-\be_\alpha\rangle+
          \langle e^*_\beta,e_\beta-\be_\beta\rangle+\langle u^*,u-\ou_{\be}\rangle}
         {\|e_\alpha-\be_\alpha\|_{L^2(\Omega)}+\|e_\beta-\be_\beta\|_{L^2(\Omega)}+\|u-\ou_{\be}\|_{L^2(\Omega)}}
    \leq0.
\end{equation}
Let $e_\alpha-\be_\alpha=e_\beta-\be_\beta=u-\ou_{\be}$. Then,
$(e,u)\in\gph\mU_{ad}$ and from \eqref{LmspFrCne} we have
$$\limsup_{e_\alpha\to\be_\alpha}
  \frac{\langle e^*_\alpha+e^*_\beta+u^*,e_\alpha-\be_\alpha\rangle}{\|e_\alpha-\be_\alpha\|_{L^2(\Omega)}}\leq0.$$
This yields $e^*_\alpha+e^*_\beta+u^*=0$, or
$u^*=-e^*_\alpha-e^*_\beta$. We now consider the following
situations.

$\bullet$ Necessary conditions for $e^*_\alpha$:
\begin{itemize}
\item[$\centerdot$] Let $e_\beta=\be_\beta$, $u=\ou_{\be}$,
$e_\alpha|_{\Omega_1(\be,\ou_{\be})}\nearrow
\be_\alpha|_{\Omega_1(\be,\ou_{\be})}$,
$e_\alpha|_{\Omega\setminus\Omega_1(\be,\ou_{\be})}
=\be_\alpha|_{\Omega\setminus\Omega_1(\be,\ou_{\be})}$. Then, we
have $(e,u)\in\gph\mU_{ad}$ and from \eqref{LmspFrCne} we deduce
that $e^*_\alpha|_{\Omega_1(\be,\ou_{\be})}\geq0$.

\item[$\centerdot$] Let $e_\beta=\be_\beta$ and $u=\ou_{\be}$.
For $\varepsilon>0$, define
$A_\varepsilon:=\{x\in\Omega\st\ou_{\be}\geq\alpha+\be_\alpha+\varepsilon\}\subset
\Omega\setminus\Omega_1(\be,\ou)$. Let now $B\subset A_\varepsilon$
with positive measure. Let us set
$e_\alpha|_{\Omega_1(\be,\ou_{\be})}=\be_\alpha|_{\Omega_1(\be,\ou_{\be})}$
and
$$e_\alpha|_{\Omega\setminus\Omega_1(\be,\ou_{\be})}=\be_\alpha|_{\Omega\setminus\Omega_1(\be,\ou_{\be})}+t\chi_B,$$
where $|t|<\varepsilon$. Then, we have $(e,u)\in\gph\mU_{ad}$, and
from \eqref{LmspFrCne} we deduce
$$0\geq\limsup_{e_\alpha\to\be_\alpha}
   \frac{\langle e^*_\alpha,e_\alpha-\be_\alpha\rangle}{\|e_\alpha-\be_\alpha\|_{L^2(\Omega)}}
   \geq\limsup_{t\to0}\frac{\langle e^*_\alpha,t\chi_B\rangle}{\|t\chi_B\|_{L^2(\Omega)}}
   =\|\chi_B\|_{L^2(\Omega)}^{-1}|\langle e^*_\alpha,\chi_B\rangle|.$$
This implies $e^*_\alpha=0$ almost everywhere on $A_\varepsilon$.
Since
$\cup_{\varepsilon>0}A_\varepsilon=\Omega\setminus\Omega_1(\be,\ou_{\be})$,
we find $e^*_\alpha=0$ almost everywhere on
$\Omega\setminus\Omega_1(\be,\ou_{\be})$ as claimed.
\end{itemize}

$\bullet$ Necessary conditions for $e^*_\beta$:
\begin{itemize}
\item[$\centerdot$] Let $e_\alpha=\be_\alpha$, $u=\ou_{\be}$,
$e_\beta|_{\Omega\setminus\Omega_3(\be,\ou_{\be})}
=\be_\beta|_{\Omega\setminus\Omega_3(\be,\ou_{\be})}$,
$e_\beta|_{\Omega_3(\be,\ou_{\be})}\searrow\be_\beta|_{\Omega_3(\be,\ou_{\be})}$.
Then, we have $(e,u)\in\gph\mU_{ad}$ and from \eqref{LmspFrCne} we
deduce that $e^*_\beta|_{\Omega_3(\be,\ou_{\be})}\leq0$.

\item[$\centerdot$] By arguing similarly as in the proof of the second
necessary condition for $e^*_\alpha$ above, we also find that
$e^*_\beta|_{\Omega\setminus\Omega_3(\be,\ou_{\be})}=0$.
\end{itemize}

Conversely, pick any $(e^*,u^*)$ from the set on the right-hand side
of \eqref{FrNrCnGpUad}. Taking any sequence
$(e_n,u_n)\to(\be,\ou_{\be})$ with $(e_n,u_n)\in\gph\mU_{ad}$, we
have to show that \eqref{LmspFrCne} holds for this sequence. For
convenience, we denote $\Omega_i=\Omega_i(\be,\ou_{\be})$ for
$i=1,2,3$. We observe that
$$\begin{aligned}
     &\langle e^*_\alpha,(e_\alpha)_n-\be_\alpha\rangle+
      \langle e^*_\beta,(e_\beta)_n-\be_\beta\rangle+\langle u^*,u_n-\ou_{\be}\rangle\\
    =&\langle e^*_\alpha|_{\Omega_1},(e_\alpha)_n|_{\Omega_1}-\be_\alpha|_{\Omega_1}\rangle+
      \langle e^*_\beta|_{\Omega_3},(e_\beta)_n|_{\Omega_3}-\be_\beta|_{\Omega_3}\rangle\\
     &-\langle(e^*_\alpha+e^*_\beta)|_{\Omega_1\cup\Omega_3},u_n|_{\Omega_1\cup\Omega_3}
               -\ou_{\be}|_{\Omega_1\cup\Omega_3}\rangle\\
    =&\langle e^*_\alpha|_{\Omega_1},(e_\alpha)_n|_{\Omega_1}-\be_\alpha|_{\Omega_1}\rangle+
      \langle e^*_\beta|_{\Omega_3},(e_\beta)_n|_{\Omega_3}-\be_\beta|_{\Omega_3}\rangle\\
     &-\langle e^*_\alpha|_{\Omega_1},u_n|_{\Omega_1}-\ou_{\be}|_{\Omega_1}\rangle
      -\langle e^*_\beta|_{\Omega_3},u_n|_{\Omega_3}-\ou_{\be}|_{\Omega_3}\rangle\\
    =&\langle e^*_\alpha|_{\Omega_1},(e_\alpha)_n|_{\Omega_1}-
              \be_\alpha|_{\Omega_1}-u_n|_{\Omega_1}+\ou_{\be}|_{\Omega_1}\rangle+
      \langle e^*_\beta|_{\Omega_3},(e_\beta)_n|_{\Omega_3}-
              \be_\beta|_{\Omega_3}-u_n|_{\Omega_3}+\ou_{\be}|_{\Omega_3}\rangle\\
     &\mbox{(with $\ou_{\be}|_{\Omega_1}=\alpha|_{\Omega_1}+\be_\alpha|_{\Omega_1}$
             and $\ou_{\be}|_{\Omega_3}=\beta|_{\Omega_3}+\be_\beta|_{\Omega_3}$ by the definition of
             $\Omega_1$ and $\Omega_3$)}\\
    =&\langle\underset{\geq0}{\underbrace{e^*_\alpha|_{\Omega_1}}},
             \underset{\leq0}{\underbrace{\alpha|_{\Omega_1}+(e_\alpha)_n|_{\Omega_1}-u_n|_{\Omega_1}}}\rangle+
      \langle\underset{\leq0}{\underbrace{e^*_\beta|_{\Omega_3}}},
             \underset{\geq0}{\underbrace{\beta|_{\Omega_3}+(e_\beta)_n|_{\Omega_3}-u_n|_{\Omega_3}}}\rangle\\
    \leq&0.
\end{aligned}$$
This implies that \eqref{LmspFrCne} holds for the sequence
$\{(e_n,u_n)\}$ chosen above. $\hfill\Box$
\end{proof}

\begin{Proposition}\label{PrpFrCdrmU}
Assume that the assumptions {\rm\textbf{(A1)}-\textbf{(A3)}} hold
and let $\ou_{\be}\in S(\be)$. Then, the following formula holds
that
\begin{equation}\label{FreCdrmU}
\begin{aligned}
    \widehat{D}^*\mG(\be,\ou_{\be})(u^*)
    &=\Big\{e^*\in E^*\Bst e^*=(0,0,e^*_\alpha,e^*_\beta),\,u^*=u^*_1-u^*_2,\\
    &\qquad u^*_1=e^*_\alpha+e^*_\beta,\,u^*_2\in N(\ou_{\be};\mQ),\\
    &\qquad e^*_\alpha|_{\Omega_1(\be,\ou_{\be})}\geq0,\,e^*_\alpha|_{\Omega\setminus\Omega_1(\be,\ou_{\be})}=0,\\
    &\qquad e^*_\beta|_{\Omega_3(\be,\ou_{\be})}\leq0,\,e^*_\beta|_{\Omega\setminus\Omega_3(\be,\ou_{\be})}=0\Big\}.
\end{aligned}
\end{equation}
\end{Proposition}
\begin{proof}
Formula~\eqref{FreCdrmU} follows directly from \eqref{FreGSmRle} and
\eqref{FrNrCnGpUad}. $\hfill\Box$
\end{proof}

\medskip
The forthcoming theorem establishes an upper estimate for the
regular subdifferential of the marginal function $\mu(\cdot)$.

\begin{Theorem}\label{ThmUpEstFrSb}
Assume that the assumptions {\rm\textbf{(A1)}-\textbf{(A3)}} hold
and let $\ou_{\be}\in S(\be)$. Then it is necessary for an element
$\he^*=(\he^*_y,\he^*_J,\he^*_\alpha,\he^*_\beta)$ from $E^*$
belonging to $\widehat{\partial}\mu(\be)$ that
\begin{equation}\label{DKFrSbDiff}
\begin{cases}
   \he^*_y=\varphi_{\ou_{\be},\be}+\zeta\ou_{\be},\\
   \he^*_J=y_{\ou_{\be}+\be_y},\\
   \he^*_\alpha|_{\Omega_1(\be,\ou_{\be})}\geq0,\,
   \he^*_\alpha|_{\Omega\setminus\Omega_1(\be,\ou_{\be})}=0,\\
   \he^*_\beta|_{\Omega_3(\be,\ou_{\be})}\leq0,\,\he^*_\beta|_{\Omega\setminus\Omega_3(\be,\ou_{\be})}=0,\\
   \he^*_\alpha+\he^*_\beta\in N(\ou_{\be};\mQ)+\varphi_{\ou_{\be},\be}+\zeta\ou_{\be}.
\end{cases}
\end{equation}
If, in addition, $S:\dom\mG\rightrightarrows L^2(\Omega)$ has a
local upper Lipschitzian selection at $(\be,\ou_{\be})$, then
condition~\eqref{DKFrSbDiff} is also sufficient for the inclusion
$\he^*\in\widehat{\partial}\mu(\be)$.
\end{Theorem}
\begin{proof}
Pick any
$\he^*=(\he^*_y,\he^*_J,\he^*_\alpha,\he^*_\beta)\in\widehat{\partial}\mu(\be)$.
Using \eqref{FrSbEstNoCp} we obtain the following inclusion
$$\he^*\in\mJ'_e(\ou_{\be},\be)+\widehat{D}^*\mG(\be,\ou_{\be})\big(\mJ'_u(\ou_{\be},\be)\big),$$
or, equivalently, as follows
\begin{equation}\label{InclFrSbGr}
    \he^*-\mJ'_e(\ou_{\be},\be)\in\widehat{D}^*\mG(\be,\ou_{\be})\big(\mJ'_u(\ou_{\be},\be)\big).
\end{equation}
From \eqref{PerProWtCd} we have
$$\mJ(u,e)=J(u+e_y)+\big(e_J,G(u+e_y)\big)_{L^2(\Omega)}.$$
It follows that
$$\mJ'_u(\ou_{\be},\be)v=J'(\ou_{\be}+\be_y)v+\big(\be_J,G'(\ou_{\be}+\be_y)v\big)_{L^2(\Omega)},$$
and
\begin{equation}\label{DermJue}
\begin{aligned}
    \mJ'_e(\ou_{\be},\be)\we
    &=J'(\ou_{\be}+\be_y)\we_y+\big(\be_J,G'(\ou_{\be}+\be_y)\we_y\big)_{L^2(\Omega)}
     +\big(\we_J,G(\ou_{\be}+\be_y)\big)_{L^2(\Omega)}\\
    &=\mJ'_u(\ou_{\be},\be)\we_y+(\we_J,y_{\ou_{\be}+\be_y})_{L^2(\Omega)},
\end{aligned}
\end{equation}
where $y_{\ou_{\be}+\be_y}=G(\ou_{\be}+\be_y)$. Using
\eqref{PerbDermJ} we find that
$\mJ'_u(\ou_{\be},\be)=\varphi_{\ou_{\be},\be}+\zeta\ou_{\be}$.
Combining this with \eqref{DermJue} and the fact that $\mJ(u,e)$
does not depend on $e_\alpha$ and $e_\beta$, we obtain
$$\mJ'_e(\ou_{\be},\be)
=\big(\varphi_{\ou_{\be},\be}+\zeta\ou_{\be},y_{\ou_{\be}+\be_y},0_{L^2(\Omega)},0_{L^2(\Omega)}\big).$$
Consequently, we have
\begin{equation}\label{HeStrmJe}
    \he^*-\mJ'_e(\ou_{\be},\be)
    =\big(\he^*_y-\varphi_{\ou_{\be},\be}-\zeta\ou_{\be},\he^*_J-y_{\ou_{\be}+\be_y},\he^*_\alpha,\he^*_\beta\big).
\end{equation}
From \eqref{HeStrmJe}, \eqref{InclFrSbGr} and \eqref{FreCdrmU} we
deduce that
\begin{equation}\label{YJAlBeStr}
    \he^*_y-\varphi_{\ou_{\be},\be}-\zeta\ou_{\be}=0,\quad\he^*_J-y_{\ou_{\be}+\be_y}=0,
    \quad\he^*_\alpha=e^*_\alpha,\quad\he^*_\beta=e^*_\beta,
\end{equation}
where $e^*_\alpha$ and $e^*_\beta$ satisfy the following condition
$$\begin{cases}
    e^*_\alpha|_{\Omega_1(\be,\ou_{\be})}\geq0,\, e^*_\alpha|_{\Omega\setminus\Omega_1(\be,\ou_{\be})}=0,\\
    e^*_\beta|_{\Omega_3(\be,\ou_{\be})}\leq0,\,e^*_\beta|_{\Omega\setminus\Omega_3(\be,\ou_{\be})}=0,\\
    \mJ'_u(\ou_{\be},\be)=u^*_1-u^*_2,\\
    u^*_1=e^*_\alpha+e^*_\beta,\,u^*_2\in N(\ou_{\be};\mQ),
\end{cases}$$
or, equivalently, as follows
\begin{equation}\label{AlBeStr}
\begin{cases}
    e^*_\alpha|_{\Omega_1(\be,\ou_{\be})}\geq0,\,e^*_\alpha|_{\Omega\setminus\Omega_1(\be,\ou_{\be})}=0,\\
    e^*_\beta|_{\Omega_3(\be,\ou_{\be})}\leq0,\,e^*_\beta|_{\Omega\setminus\Omega_3(\be,\ou_{\be})}=0,\\
    \varphi_{\ou_{\be},\be}+\zeta\ou_{\be}=e^*_\alpha+e^*_\beta-u^*_2,\\
    u^*_2\in N(\ou_{\be};\mQ).
\end{cases}
\end{equation}
Combining \eqref{YJAlBeStr} with \eqref{AlBeStr} we get
$$\begin{cases}
   \he^*_y=\varphi_{\ou_{\be},\be}+\zeta\ou_{\be},\\
   \he^*_J=y_{\ou_{\be}+\be_y},\\
   \he^*_\alpha|_{\Omega_1(\be,\ou_{\be})}\geq0,\,
   \he^*_\alpha|_{\Omega\setminus\Omega_1(\be,\ou_{\be})}=0,\\
   \he^*_\beta|_{\Omega_3(\be,\ou_{\be})}\leq0,\,\he^*_\beta|_{\Omega\setminus\Omega_3(\be,\ou_{\be})}=0,\\
   \he^*_y=\he^*_\alpha+\he^*_\beta-u^*_2,\,u^*_2\in N(\ou_{\be};\mQ),
\end{cases}$$
which yields \eqref{DKFrSbDiff}.

If, in addition, $S:\dom\mG\rightrightarrows L^2(\Omega)$ admits a
local upper Lipschitzian selection at $(\be,\ou_{\be})$, then in the
arguments above we use equality~\eqref{FreSubEqul} instead of
estimate~\eqref{FrSbEstNoCp} to deduce that
condition~\eqref{DKFrSbDiff} is necessary and sufficient for the
inclusion $\he^*\in\widehat{\partial}\mu(\be)$. $\hfill\Box$
\end{proof}

\subsection{Limiting subgradients of marginal functions}

The next proposition provides us with an explicit formula for
computing the Mordukhovich coderivative of the multifunction
$\mG(\cdot)$ that will be used to establish upper estimates for the
Mordukhovich and the singular subdifferentials of the marginal
function $\mu(\cdot)$.

\begin{Proposition}\label{PropMorCdrG}
Assume that the assumptions {\rm\textbf{(A1)}-\textbf{(A3)}} hold
and let $\ou_{\be}\in S(\be)$. Then, for every $u^*\in L^2(\Omega)$,
we have
\begin{equation}\label{MorCdrmU}
\begin{aligned}
    D^*\mG(\be,\ou_{\be})(u^*)
    &=\widehat{D}^*\mG(\be,\ou_{\be})(u^*)\\
    &=\Big\{e^*\in E^*\Bst e^*=(0,0,e^*_\alpha,e^*_\beta),\,u^*=u^*_1-u^*_2,\\
    &\qquad u^*_1=e^*_\alpha+e^*_\beta,\,u^*_2\in N(\ou_{\be};\mQ),\\
    &\qquad e^*_\alpha|_{\Omega_1(\be,\ou_{\be})}\geq0,\,e^*_\alpha|_{\Omega\setminus\Omega_1(\be,\ou_{\be})}=0,\\
    &\qquad e^*_\beta|_{\Omega_3(\be,\ou_{\be})}\leq0,\,e^*_\beta|_{\Omega\setminus\Omega_3(\be,\ou_{\be})}=0\Big\}.
\end{aligned}
\end{equation}
\end{Proposition}
\begin{proof}
We observe that $\gph\mG$ is closed in the product space $E\times
L^2(\Omega)$. By the definitions of coderivatives, we have
$\widehat{D}^*\mG(\be,\ou_{\be})(u^*)\subset
D^*\mG(\be,\ou_{\be})(u^*)$. Let us verify the opposite inclusion.
Fix any $e^*=(e^*_y,e^*_J,e^*_\alpha,e^*_\beta)\in
D^*\mG(\be,\ou_{\be})(u^*)$. Then, by \eqref{DefFrDrvtv},
\eqref{DefMrDrvtv}, and \eqref{LmNrCnFrCn}, there exist sequences
$(e_n,u_n)\in\gph\mG$ and $(e^*_n,u^*_n)\in E^*\times L^2(\Omega)$
satisfying
$$(e_n,u_n)\to(\be,\ou_{\be}),~(e^*_n,u^*_n)\stackrel{w^*}\rightharpoonup
  (e^*,u^*),~e^*_n\in\widehat{D}^*\mG(e_n,u_n)(u^*_n),~\forall n\in\N$$
and
$$(e_n,u_n)\to(\be,\ou_{\be})~\mbox{pointwise a.e. on}~\Omega.$$
For every $n\in\N$, since
$e^*_n\in\widehat{D}^*\mG(e_n,u_n)(u^*_n)$, by \eqref{FreCdrmU} we
infer that $e^*_n=\big(0,0,(e^*_n)_\alpha,(e^*_n)_\beta\big)$
satisfies the following conditions
\begin{equation}\label{CndFreMor}
\begin{cases}
    u^*_n=(u^*_n)_1-(u^*_n)_2,\\
    (u^*_n)_1=(e^*_n)_\alpha+(e^*_n)_\beta,\,(u^*_n)_2\in N(u_n;\mQ),\\
    (e^*_n)_\alpha|_{\Omega_1(e_n,u_n)}\geq0,\,(e^*_n)_\alpha|_{\Omega\setminus\Omega_1(e_n,u_n)}=0,\\
    (e^*_n)_\beta|_{\Omega_3(e_n,u_n)}\leq0,\,(e^*_n)_\beta|_{\Omega\setminus\Omega_3(e_n,u_n)}=0.
\end{cases}
\end{equation}
By
$e^*_n=\big(0,0,(e^*_n)_\alpha,(e^*_n)_\beta\big)\stackrel{w^*}\rightharpoonup
e^*$, we deduce that $e^*=(0,0,e^*_\alpha,e^*_\beta\big)$ with
$$(e^*_n)_\alpha\stackrel{w^*}\rightharpoonup e^*_\alpha\quad\mbox{and}\quad
(e^*_n)_\beta\stackrel{w^*}\rightharpoonup e^*_\beta.$$ From this
and \eqref{CndFreMor} it follows that $e^*_\alpha\ge0$ and
$e^*_\beta\le0$ on $\Omega$. We show that
$e^*_\alpha|_{\Omega\setminus\Omega_1(\be,\ou_{\be})}=0$ and
$e^*_\beta|_{\Omega\setminus\Omega_3(\be,\ou_{\be})}=0$. Let
$\varepsilon>0$ be given. Let $B\subset A_\varepsilon:=\{x\in
\Omega\st\ou\ge\alpha+\be_\alpha+\varepsilon\}\subset
\Omega\setminus\Omega_1(\be,\ou_{\be})$ be a bounded set of positive
measure. Since $(e^*_n)_\alpha=0$ on
$\Omega\setminus\Omega_1(\be,\ou_{\be})$, we get
$$\begin{aligned}
    \langle e^*_\alpha,\chi_B\rangle
    &=\lim_{n\to\infty}\langle(e^*_n)_\alpha,\chi_B\rangle\\
    &=\lim_{n\to\infty}\langle(e^*_n)_\alpha,\chi_B|_{\Omega\setminus\Omega_1(e_n,u_n)}
     +\chi_B|_{\Omega_1(e_n,u_n)}\rangle\\
    &=\lim_{n\to\infty}\langle(e^*_n)_\alpha,\chi_B|_{\Omega_1(e_n,u_n)}\rangle.
\end{aligned}$$
Due to pointwise convergence, we have $\chi_B
\chi_{\Omega_1(e_n,u_n)} \to 0$ pointwise almost everywhere. By
dominated convergence theorem, $\chi_B \chi_{\Omega_1(e_n,u_n)} \to
0$ in $L^2(\Omega)$. Hence, $\lim_{n\to\infty} \langle
(e^*_n)_\alpha, \chi_B \rangle=0$. It follows that $e^*_\alpha= 0$
on $A_\varepsilon$ for all $\varepsilon>0$, which in turn implies
$e^*_\alpha= 0$ on $\Omega\setminus\Omega_1(\be,\ou_{\be})$.
Similarly, we can prove
$e^*_\beta|_{\Omega\setminus\Omega_3(\be,\ou_{\be})}=0$. We have
shown that
\begin{equation}\label{CndEAlBe}
\begin{cases}
    e^*_\alpha|_{\Omega_1(\be,\ou_{\be})}\geq0,\,e^*_\alpha|_{\Omega\setminus\Omega_1(\be,\ou_{\be})}=0,\\
    e^*_\beta|_{\Omega_3(\be,\ou_{\be})}\leq0,\,e^*_\beta|_{\Omega\setminus\Omega_3(\be,\ou_{\be})}=0.
\end{cases}
\end{equation}
Since $u_n\to\ou_{\be}$ with $u_n\in\mQ$, we have
$e^*_\alpha+e^*_\beta-u^*\in N(\ou_{\be};\mQ)$. Indeed, for all
$v\in\mQ$, due to
$$(e^*_n)_\alpha+(e^*_n)_\beta-u^*_n=(u^*_n)_2\in N(u_n;\mQ),$$
we obtain
$\langle(e^*_n)_\alpha+(e^*_n)_\beta-u^*_n,v-u_n\rangle\leq0$ for
every $n\in\N$. In addition, since
$$(e^*_n)_\alpha+(e^*_n)_\beta-u^*_n\stackrel{w^*}\rightharpoonup
e^*_\alpha+e^*_\beta-u^*\quad\mbox{and}\quad v-u_n\to v-\ou_{\be},$$
we have
$$\langle(e^*_n)_\alpha+(e^*_n)_\beta-u^*_n,v-u_n\rangle\to\langle e^*_\alpha+e^*_\beta-u^*,v-\ou_{\be}\rangle.$$
This implies that $\langle
e^*_\alpha+e^*_\beta-u^*,v-\ou_{\be}\rangle\leq0$, which yields
$e^*_\alpha+e^*_\beta-u^*\in N(\ou_{\be};\mQ)$. We now put
$$u^*_1=e^*_\alpha+e^*_\beta\quad\mbox{and}\quad u^*_2=u^*_1-u^*\in N(\ou_{\be};\mQ).$$
From this and \eqref{CndEAlBe} we obtain
$e^*\in\widehat{D}^*\mG(\be,\ou_{\be})(u^*)$. Thus,
$D^*\mG(\be,\ou_{\be})(u^*)\subset\widehat{D}^*\mG(\be,\ou_{\be})(u^*)$.
We have shown that
$D^*\mG(\be,\ou_{\be})(u^*)=\widehat{D}^*\mG(\be,\ou_{\be})(u^*)$
and \eqref{MorCdrmU} follows. $\hfill\Box$
\end{proof}

\medskip
The next two theorems establish respectively upper estimates for the
Mordukhovich and the singular subdifferentials of the marginal
function $\mu(\cdot)$.

\begin{Theorem}\label{ThmEsCmpMrSb}
Assume that the assumptions {\rm\textbf{(A1)}-\textbf{(A3)}} hold
and let $\ou_{\be}\in S(\be)$. Then it is necessary for an element
$e^*=(e^*_y,e^*_J,e^*_\alpha,e^*_\beta)$ from $E^*$ belonging to
$\partial\mu(\be)$ that
\begin{equation}\label{DKMrSbDiff}
\begin{cases}
    e^*_y=\varphi_{\ou_{\be},\be}+\zeta\ou_{\be},\\
    e^*_J=y_{\ou_{\be}+\be_y},\\
    e^*_\alpha|_{\Omega_1(\be,\ou_{\be})}\geq0,\,e^*_\alpha|_{\Omega\setminus\Omega_1(\be,\ou_{\be})}=0,\\
    e^*_\beta|_{\Omega_3(\be,\ou_{\be})}\leq0,\,e^*_\beta|_{\Omega\setminus\Omega_3(\be,\ou_{\be})}=0,\\
    e^*_\alpha+e^*_\beta\in N(\ou_{\be};\mQ)+\varphi_{\ou_{\be},\be}+\zeta\ou_{\be}.
\end{cases}
\end{equation}
If, in addition, $S:\dom\mG\rightrightarrows L^2(\Omega)$ admits a
local upper Lipschitzian selection at $(\be,\ou_{\be})$, then the
marginal function $\mu(\cdot)$ is lower regular at $\be$ and
\eqref{DKMrSbDiff} is also sufficient for the inclusion
$e^*\in\partial\mu(\be)$.
\end{Theorem}
\begin{proof}
By our assumptions, $\mJ(u,e)$ is continuously differentiable at
$(\ou_{\be},\be)$, thus $\mJ(u,e)$ is strictly differentiable at
$(\ou_{\be},\be)$ and Lipschitz continuous around $(\ou_{\be},\be)$.
This implies that
$\partial\mJ(\ou_{\be},\be)=\big\{\big(\mJ'_u(\ou_{\be},\be),\mJ'_e(\ou_{\be},\be)\big)\big\}$.
Applying \cite[Theorem~7(i)]{MorNmYn09MP}, we obtain
\begin{equation}\label{UpMrSbMrcdr}
    \partial\mu(\be)\subset\mJ'_e(\ou_{\be},\be)+D^*\mG(\be,\ou_{\be})\big(\mJ'_u(\ou_{\be},\be)\big).
\end{equation}
By Proposition~\ref{PropMorCdrG}, we infer that
$D^*\mG(\be,\ou_{\be})\big(\mJ'_u(\ou_{\be},\be)\big)=\widehat{D}^*\mG(\be,\ou_{\be})\big(\mJ'_u(\ou_{\be},\be)\big)$.
From this and \eqref{UpMrSbMrcdr} we get
\begin{equation}\label{UpMrSbFrCdr}
    \partial\mu(\be)\subset\mJ'_e(\ou_{\be},\be)+\widehat{D}^*\mG(\be,\ou_{\be})\big(\mJ'_u(\ou_{\be},\be)\big).
\end{equation}
By \eqref{UpMrSbFrCdr} and by Theorem~\ref{ThmUpEstFrSb}, we deduce
that \eqref{DKMrSbDiff} is necessary for $e^*\in\partial\mu(\be)$.

If, in addition, $S:\dom\mG\rightrightarrows L^2(\Omega)$ admits a
local upper Lipschitzian selection at $(\be,\ou_{\be})$, then by
\cite[Theorem~7(iii)]{MorNmYn09MP} the marginal function
$\mu(\cdot)$ is lower regular at $\be$ and \eqref{UpMrSbMrcdr} holds
as equality
\begin{equation}\label{EqMrSbMrcdr}
\begin{aligned}
    \partial\mu(\be)
    &=\mJ'_e(\ou_{\be},\be)+D^*\mG(\be,\ou_{\be})\big(\mJ'_u(\ou_{\be},\be)\big)\\
    &=\mJ'_e(\ou_{\be},\be)+\widehat{D}^*\mG(\be,\ou_{\be})\big(\mJ'_u(\ou_{\be},\be)\big).
\end{aligned}
\end{equation}
By \eqref{EqMrSbMrcdr} and by Theorem~\ref{ThmUpEstFrSb},
condition~\eqref{DKMrSbDiff} is also sufficient for
$e^*\in\partial\mu(\be)$. $\hfill\Box$
\end{proof}

\begin{Remark}\rm
From Theorems~\ref{ThmUpEstFrSb}, \ref{ThmEsCmpMrSb} we see that the
necessary conditions~\eqref{DKFrSbDiff} and \eqref{DKMrSbDiff}
coincide because $\mG(\cdot)$ is normally regular at
$(\be,\ou_{\be})$ by \eqref{MorCdrmU}. These necessary conditions
are also the same sufficient conditions provided that
$S:\dom\mG\rightrightarrows L^2(\Omega)$ admits a local upper
Lipschitzian selection at $(\be,\ou_{\be})$, which yields
$\widehat{\partial}\mu(\be)=\partial\mu(\be)$, i.e., the marginal
function $\mu(\cdot)$ is lower regular at $\be$.
\end{Remark}

\begin{Theorem}\label{ThmEsSglrSb}
Assume that the assumptions {\rm\textbf{(A1)}-\textbf{(A3)}} hold
and let $\ou_{\be}\in S(\be)$. Then, we have the following estimate
\begin{equation}\label{UpEsSglrSb}
\begin{aligned}
    \partial^\infty\mu(\be)\subset
    &\Big\{(0,0,e^*_\alpha,e^*_\beta)\in E^*\Bst
           e^*_\alpha+e^*_\beta\in N(\ou_{\be};\mQ),\\
    &\quad e^*_\alpha|_{\Omega_1(\be,\ou_{\be})}\geq0,\,e^*_\alpha|_{\Omega\setminus\Omega_1(\be,\ou_{\be})}=0,\\
    &\quad e^*_\beta|_{\Omega_3(\be,\ou_{\be})}\leq0,\,e^*_\beta|_{\Omega\setminus\Omega_3(\be,\ou_{\be})}=0\Big\}.
\end{aligned}
\end{equation}
\end{Theorem}
\begin{proof}
Our assumptions ensure that $\mJ(u,e)$ is Lipschitz continuous
around $(\ou_{\be},\be)$. Hence, we have
$\partial^\infty\mJ(\ou_{\be},\be)=\{(0,0)\}$. By
\cite[Theorem~7(i)]{MorNmYn09MP}, we deduce that
\begin{equation}\label{EsSgSbMrCdr}
    \partial^\infty\mu(\be)\subset D^*\mG(\be,\ou_{\be})(0).
\end{equation}
By \eqref{EsSgSbMrCdr}, formula~\eqref{UpEsSglrSb} follows directly
from formula~\eqref{MorCdrmU}. $\hfill\Box$
\end{proof}

\begin{Corollary}
Assume that the assumptions {\rm\textbf{(A1)}-\textbf{(A3)}} hold
and let $\ou_{\be}\in S(\be)$. Then, the following hold:
\begin{itemize}
\item[{\rm(i)}] If $\ou_{\be}\in\inter\mQ$, then we have
\begin{equation}\label{UpEsSgSb0}
    \partial^\infty\mu(\be)\subset\{0\}.
\end{equation}
\item[{\rm(ii)}] If $\ou_{\be}\in\inter\mQ$, and there exists a sequence $e_n\to\be$ such that
$\ou_{e_n}\to\ou_{\be}$ in $L^{p_0}(\Omega)$ with $\ou_{e_n}\in
S(e_n)$ and $\widehat{\partial}\mu(e_n)\neq\emptyset$, then we have
\begin{equation}\label{LwEsSgSb0}
    0\in\partial^\infty\mu(\be).
\end{equation}
Consequently, \eqref{UpEsSgSb0} holds as an equality.
\end{itemize}
\end{Corollary}
\begin{proof}
(i) Take any $(0,0,e^*_\alpha,e^*_\beta)\in\partial^\infty\mu(\be)$.
Note that $N(\ou_{\be};\mQ)=\{0\}$ because $\ou_{\be}\in\inter\mQ$.
Hence, from \eqref{UpEsSglrSb} it follows that
$e^*_\alpha=e^*_\beta=0$ a.e. on $\Omega$. This yields
\eqref{UpEsSgSb0}.

(ii) Choose $\lambda_n=\varepsilon_n=1/n$ and take any
$\he^*_n\in\widehat{\partial}\mu(e_n)\subset\widehat{\partial}_{\varepsilon_n}\mu(e_n)$
for every $n\in\N$. Since $\he^*_n\in\widehat{\partial}\mu(e_n)$,
$\he^*_n$ holds \eqref{DKFrSbDiff}. Because $\ou_{e_n}\to\ou_{\be}$
and $\ou_{\be}\in\inter\mQ$, we have $\ou_{e_n}\in\inter\mQ$ for all
$n$ large enough. Hence, $N(\ou_{e_n};\mQ)=\{0\}$ for all $n$
sufficiently large. Consequently, according to \eqref{DKFrSbDiff},
$\he^*_n$ must be bounded. Letting $n\to\infty$, we have
\begin{equation}\label{CdLwEsSgSb}
    e_n\to\be,~~\mu(e_n)\to\mu(\be),~~\varepsilon_n\downarrow0,~~
    \lambda_n\downarrow0,~~\lambda_n\he^*_n\stackrel{w^*}\rightharpoonup0,
\end{equation}
which yields $0\in\partial^\infty\mu(\be)$ by \eqref{SngSubdif}.
Combining this with (i) we obtain $\partial^\infty\mu(\be)=\{0\}$.
$\hfill\Box$
\end{proof}

\begin{Remark}\rm
If $\widehat{\partial}\mu(\be)\neq\emptyset$, then \eqref{LwEsSgSb0}
holds without the assumption $\ou_{\be}\in\inter\mQ$. Indeed, take
any $\he^*\in\widehat{\partial}\mu(\be)$ and choose $e_n=\be$,
$\lambda_n=\varepsilon_n=1/n$, $\he^*_n=\he^*$ for every $n\in\N$.
Letting $n\to\infty$, we obtain \eqref{CdLwEsSgSb}, which implies
$0\in\partial^\infty\mu(\be)$.
\end{Remark}

\section{Parametric bang-bang control problems}
\setcounter{equation}{0}

In this section, we consider the parametric control
problem~\eqref{PerProWtCd}, where the functional $J(\cdot)$ is given
in \eqref{OptConPro} with $\zeta=0$ a.e. on $\Omega$,
$\mQ=L^{p_0}(\Omega)$ with $p_0>N/2$ while $q_0=s_0=1$, and
$p_1=p_2=2$, $p_3=p_4=\infty$ in \eqref{DefNormE}. In addition, the
solution map $S:E\rightrightarrows L^1(\Omega)$ is defined by
\eqref{SolMap} with respect to $\mJ(u,e):L^1(\Omega)\times E\to\R$.
For this setting, we rewrite problem~\eqref{PerProWtCd} as follows
\begin{equation}\label{BBConProb}
\begin{cases}
    {\rm Minimize}   & \mJ(u,e)=J(u+e_y)+(e_J,y_{u+e_y})_{L^2(\Omega)}\\
    {\rm subject~to} & u\in\mU_{ad}(e),
\end{cases}
\end{equation}
where $y_{u+e_y}$ is the weak solution of \eqref{PerStaEqWtCd}, and
the functional $J(\cdot)$ is defined by
$$J(u)=\displaystyle\int_\Omega L\big(x,y_u(x)\big)dx.$$
Note that we have $\mU_{ad}(e)\subset L^\infty(\Omega)$ for every
$e\in E=L^2(\Omega)\times L^2(\Omega)\times L^\infty(\Omega)\times
L^\infty(\Omega)$.

In contrast to the previous section, the cost functional
$\mJ:L^1(\Omega)\times E\to\R$ of problem~\eqref{BBConProb} is not
Fr\'echet differentiable. In addition, the problem of computing
$\widehat{\partial}^+\mJ(\ou_{\be},\be)$ or checking
$\widehat{\partial}^+\mJ(\ou_{\be},\be)\neq\emptyset$ at a given
point $(\ou_{\be},\be)\in\gph S$ remains open. Therefore, we can not
apply \cite[Theorems~1 and 2]{MorNmYn09MP} to compute/estimate
subdifferentials of the marginal function $\mu(\cdot)$ of
problem~\eqref{BBConProb}. For this reason, by the definition of
regular subgradients we will establish directly a characterization
of regular subgradients of the marginal function $\mu(\cdot)$ at a
given point $(\be,\ou_{\be})\in\gph S$ in a subspace $E^*_1$ (see
the definition of $E^*_1$ below) of $E^*$ via local upper
H\"{o}lderian selections of the solution map
$S:\dom\mG\rightrightarrows L^1(\Omega)$ at the point
$(\be,\ou_{\be})$. This leads to some lower estimates for the
regular subdifferential of $\mu(\cdot)$ at $(\be,\ou_{\be})$.

Consider problem~\eqref{OptConPro} with $\mU_{ad}$ being replaced by
$\mU_{ad}(\be)$ and let $\ou_{\be}\in\mU_{ad}(\be)$ be a solution of
problem~\eqref{OptConPro} in the sense of $L^{p_0}(\Omega)$. From
\eqref{VarIneq}, we deduce that
\begin{equation}\label{ValOuVarp}
\ou_{\be}(x)=\begin{cases}
            (\alpha+\be_\alpha)(x),  &\mbox{if}\ \varphi_{\ou_{\be}}(x)>0\\
            (\beta+\be_\beta)(x),    &\mbox{if}\ \varphi_{\ou_{\be}}(x)<0,
\end{cases}
\end{equation}
and
\begin{equation}\label{ValVarpOu}
\varphi_{\ou_{\be}}(x)\begin{cases}
            \geq0,   &\mbox{if}\ \ou_{\be}(x)=(\alpha+\be_\alpha)(x)\\
            \leq0,   &\mbox{if}\ \ou_{\be}(x)=(\beta+\be_\beta)(x)\\
            =0,      &\mbox{if}\ \ou_{\be}(x)\in\big((\alpha+\be_\alpha)(x),(\beta+\be_\beta)(x)\big).
\end{cases}
\end{equation}
In general, solutions $\ou_{\be}$ have the so-called \emph{bang-bang
property}: for a.a. $x\in\Omega$, it holds that
$\ou_{\be}(x)\in\{(\alpha+\be_\alpha)(x),(\beta+\be_\beta)(x)\}$.
Consider the case where $\{x\in\Omega\st\varphi_{\ou_{\be}}(x)=0\}$
has a zero Lebesgue measure. Then, it follows from \eqref{ValOuVarp}
and \eqref{ValVarpOu} that
$\ou_{\be}(x)\in\{\alpha+\be_\alpha)(x),(\beta+\be_\beta)(x)\}$ for
a.a. $x\in\Omega$, i.e., $\ou_{\be}$ is a bang-bang control. In this
section, we are interested in the last property of the reference
control $\ou_{\be}$.

\subsection{Local upper H\"{o}lderian selections of solution map}

According to \cite{CaWaWa17SICON}, sufficient second-order
optimality conditions for bang-bang controls $\ou_{\be}$ of
problem~\eqref{OptConPro} with respect to the admissible control set
$\mU_{ad}(\be)$ established under the following assumption
\textbf{(A4)} posed on the adjoint state $\varphi_{\ou_{\be}}$ in
the case where $\{x\in\Omega\st\varphi_{\ou_{\be}}(x)=0\}$ has a
zero Lebesgue measure.

\textbf{(A4)} Assume that $\ou_{\be}\in\mU_{ad}(\be)$, and it
satisfies the system \eqref{StateEqSol}-\eqref{VarIneq} and the
following condition
\begin{equation}\label{AsmAdVrphi}
    \exists K>0,\exists\ae>0\ \mbox{such that}\ \big|\{x\in\Omega:|\varphi_{\ou_{\be}}(x)|\leq\varepsilon\}\big|
    \leq K\varepsilon^{\ae},\ \forall\varepsilon>0.
\end{equation}
In \eqref{AsmAdVrphi}, the notation $|\cdot|$ stands for the
Lebesgue measure.

\begin{Proposition}{\rm(See \cite[Proposition~3.2]{QuiWch17})}\label{PropFstCd}
Assume that {\rm\textbf{(A1)}-\textbf{(A4)}} hold at
$\ou_{\be}\in\mU_{ad}(\be)$. Then, there exists $\kappa>0$ such that
\begin{equation}\label{FOrdL1}
    J'(\ou_{\be})(u-\ou_{\be})\geq\kappa\|u-\ou_{\be}\|^{1+\frac{1}{\ae}}_{L^1(\Omega)},\ \forall u\in\mU_{ad}(\be).
\end{equation}
\end{Proposition}

For each $\ou_{\be}\in\mU_{ad}(\be)$ and $\tau\geq0$, $1\leq
p\leq\infty$, we define the cone
\begin{equation}\label{ExCriConce}
C^\tau_{\ou_{\be},p}=\left\{v\in L^p(\Omega)\Bigg\st v(x)
\begin{cases}
    \geq0\quad\mbox{if}\ \ou_{\be}(x)=(\alpha+\be_\alpha)(x)\\
    \leq0\quad\mbox{if}\ \ou_{\be}(x)=(\beta+\be_\beta)(x)\\
       =0\quad\mbox{if}\ |\varphi_{\ou_{\be}}(x)|>\tau
\end{cases}\right\}.
\end{equation}
The forthcoming theorem provides us with a second-order sufficient
condition for bang-bang controls $\ou_{\be}\in\mU_{ad}(\be)$ to be
optimal for problem~\eqref{BBConProb} with respect to $\be\in E$.

\begin{Theorem}{\rm(See \cite[Theorem~3.1]{QuiWch17})}\label{ThmSSC}
Let $\ou_{\be}\in\mU_{ad}(\be)$ be a feasible bang-bang control for
problem~\eqref{BBConProb} satisfying
{\rm\textbf{(A1)}-\textbf{(A4)}}. Assume that there exist $\delta>0$
and $\tau>0$ such that
\begin{equation}\label{SOrdCd}
    J''(\ou_{\be})v^2\geq\delta\|z_v\|^2_{L^2(\Omega)},~\forall v\in C^\tau_{\ou_{\be},2},
\end{equation}
where $z_v=G'(\ou_{\be})v$ is the solution of \eqref{EqSolZuv} for
$y=y_{\ou_{\be}}$. Then, there exists $\varepsilon>0$ such that
\begin{equation}\label{GrOpCd}
    J(\ou_{\be})+\frac{\kappa}{2}\|u-\ou_{\be}\|^{1+\frac{1}{\ae}}_{L^1(\Omega)}
                +\frac{\delta}{8}\|z_{u-\ou_{\be}}\|^2_{L^2(\Omega)}
    \leq J(u),\ \forall u\in\mU_{ad}(\be)\cap\oB^2_\varepsilon(\ou_{\be}),
\end{equation}
with $z_{u-\ou_{\be}}=G'(\ou_{\be})(u-\ou_{\be})$ and $\kappa$ being
given in Proposition~\ref{PropFstCd}.
\end{Theorem}

In what follows, assumption~\textbf{(A4)} and
condition~\eqref{SOrdCd} will play a crucial role to prove the
existence of local upper H\"{o}lderian selections of the solution
map $S(\cdot)$.

Recall that a vector $v$ is an \emph{extremal point} of a set
$\Theta$ in a Banach space $X$ if and only if $v=\lambda
v_1+(1-\lambda)v_2$ with $v_1,v_2\in\Theta$ and $0<\lambda<1$
entails $v_1=v_2=v$. We will denote the closed convex hull of
$\Theta$ by $\overline{\conv}\,\Theta$.

\begin{Theorem}\label{ThmVisinL1}{\rm(See \cite[Theorem~1]{Visin84CPDE})}
Assume that $u_n\rightharpoonup u$ in $L^1(\Omega)$ and $u(x)$ is an
extremal point of
$\Theta(x):=\overline{\conv}\big(\{u_n(x)\}_{n\in\N}\cup\{u(x)\}\big)$
for a.a. $x\in\Omega$. Then, $u_n\to u$ in $L^1(\Omega)$.
\end{Theorem}

We will use this theorem to lift weak convergence to strong convergence.

\begin{Lemma}\label{LemBBPrp}
 Let $\ou_{\be}$ be bang-bang, i.e.,  $\ou_{\be}(x) \in \{\alpha(x) + \be_\alpha(x), \beta(x) + \be_\beta(x)\}$
 for almost all $x\in \Omega$.
 Let $e_n\to\be$ in $E$ and choose $u_n \in \mU_{ad}(e_n)$ such that $u_n \rightharpoonup \ou_{\be}$ in $L^1(\Omega)$.
 Then $u_n \to \ou_{\be}$ in $L^1(\Omega)$.
\end{Lemma}
\begin{proof}
On the active set $\Omega_1(\be,\ou_\be)$ it holds
$\ou_\be=\alpha+\be_\alpha$, cf., \eqref{DjntSetsOm}, which implies
$u_n-\ou_\be-((e_\alpha)_n-\be_\alpha) \ge0$ on this subset. In
addition, $u_n-\ou_\be-((e_\alpha)_n-\be_\alpha)\rightharpoonup0$ in
$L^1(\Omega)$. Then by Theorem~\ref{ThmVisinL1} we conclude
$u_n-\ou_\be-((e_\alpha)_n-\be_\alpha)\to0$ in
$L^1(\Omega_1(\be,\ou_\be))$, which implies $u_n\to\ou_\be$  in
$L^1(\Omega_1(\be,\ou_\be))$. Similarly, we find $u_n\to\ou_\be$ in
$L^1(\Omega_3(\be,\ou_\be))$. Since $\ou_\be$ is bang-bang, it holds
$\Omega=\Omega_1(\be,\ou_\be)\cup\Omega_3(\be,\ou_\be)$, which
proves the claim. $\hfill\Box$
\end{proof}

\medskip
A straightforward application of the above Theorem \ref{ThmVisinL1}
would require to assume that $\ou_{\be}(x)$ is an extremal point of
the set
$\overline{\conv}\big(\{u_n(x)\}_{n\in\N}\cup\{\ou_\be(x)\}\big)$
for a.a. $x\in\Omega$. This cannot be guaranteed as the control
bounds are perturbed, so $\ou_{\be}(x)=\alpha(x)+\be_\alpha(x)$ does
not imply $\ou_{\be}(x)\leq u_n(x)$.

Note that similarly to Theorem~\ref{ThmAxExSol}, under the
assumptions {\rm\textbf{(A1)}-\textbf{(A3)}} we can show that for
any $\ou_{\be}\in S(\be)$ and for every $e\in E$ near $\be$ enough
the problem of minimizing the cost functional $\mJ(u,e)$ subject to
$u\in\mU_{ad}(e)\cap\oB^{p_0}_{\varepsilon}(\ou_{\be})$ has at least
one global solution $\ou_e$, where
$\oB^{p_0}_{\varepsilon}(\ou_{\be})$ is the closed ball of center
$\ou_{\be}$ and radius $\varepsilon>0$ in $L^{p_0}(\Omega)$.

\begin{Theorem}\label{ThmSolSbL2}
Assume that {\rm\textbf{(A1)}-\textbf{(A3)}} hold and let
$\ou_{\be}\in\mU_{ad}(\be)$ be a bang-bang solution of
problem~\eqref{BBConProb} with respect to $\be\in E$ such that
$\ou_{\be}$ is strict in some neighborhood
$\oB^{p_0}_{\varepsilon}(\ou_{\be})$ with $\varepsilon>0$. For every
$e\in E$ near $\be$ enough, let $\ou_e$ be a solution of the
following control problem
\begin{equation}\label{PrPrWtCdGl}
    {\rm Minimize}\quad\mJ(u,e)\quad
    {\rm subject~to}\quad u\in\mU_{ad}(e)\cap\oB^{p_0}_{\varepsilon}(\ou_{\be}),
\end{equation}
where $\mJ(\cdot,\cdot)$ is the cost functional of
problem~\eqref{BBConProb}. Then, we have $\ou_e\to\ou_{\be}$ in
$L^{p_0}(\Omega)$ when $e\to\be$ in $E$.
\end{Theorem}
\begin{proof}
Let $\{e_n\}_{n\in\N}$ be such that $e_n\to\be$ in $E$ and let
$\ou_{e_n}\in\mU_{ad}(e_n)\cap\oB^{p_0}_{\varepsilon}(\ou_{\be})$ be
a global solution of problem~\eqref{PrPrWtCdGl} with respect to
$e_n$. Since the sequence $\{\ou_{e_n}\}$ is bounded in
$L^{p_0}(\Omega)$, it has a subsequence $\{\ou_{e_{n_k}}\}$ with
$\ou_{e_{n_k}}\rightharpoonup\hu$ in $L^{p_0}(\Omega)$ for some
$\hu\in\mU_{ad}(\be)\cap\oB^{p_0}_{\varepsilon}(\ou_{\be})$. Because
$\ou_{\be}\in\mU_{ad}(\be)$, we have
$$\ou_{\be}=\lambda(\alpha+\be_{\alpha})+(1-\lambda)(\beta+\be_{\beta})$$
for some $\lambda(x)\in[0,1]$ for almost all $x\in\Omega$. Defining
$u_{e_{n_k}}\in\mU_{ad}(e_{n_k})\cap\oB^{p_0}_{\varepsilon}(\ou_{\be})$
by
$$u_{e_{n_k}}:=\lambda\big(\alpha+(e_{n_k})_{\alpha}\big)+(1-\lambda)\big(\beta+(e_{n_k})_{\beta}\big)$$
for almost all $x\in\Omega$, we have
$$\|u_{e_{n_k}}-\ou_{\be}\|_{L^\infty(\Omega)}
  \leq\|(e_{n_k})_{\alpha}-\be_{\alpha}\|_{L^\infty(\Omega)}+\|(e_{n_k})_{\beta}-\be_{\beta}\|_{L^\infty(\Omega)}
  \leq\|e_{n_k}-\be\|_E.$$
It follows that when $k\to\infty$, we have
$$\|u_{e_{n_k}}-\ou_{\be}\|_{L^{p_0}(\Omega)}\leq|\Omega|^{1/p_0}\|e_{n_k}-\be\|_E\to0.$$
Letting $k\to\infty$, we get
$\mJ(\ou_{e_{n_k}},e_{n_k})\to\mJ(\hu,\be)$ and
$\mJ(u_{e_{n_k}},e_{n_k})\to\mJ(\ou_{\be},\be)$ with
$$\mJ(\ou_{e_{n_k}},e_{n_k})\leq\mJ(u_{e_{n_k}},e_{n_k}),~\forall k\in\N.$$
This yields $\mJ(\hu,\be)\leq\mJ(\ou_{\be},\be)$. Therefore, we
obtain $\hu=\ou_{\be}$ since $\ou_{\be}$ is a strict local solution
of problem~\eqref{BBConProb} with respect to $\be$. We have shown
that $\ou_{e_{n_k}}\rightharpoonup\ou_{\be}$ in $L^{p_0}(\Omega)$.
Consequently, $\ou_{e_{n_k}}\rightharpoonup\ou_{\be}$ in
$L^1(\Omega)$. Since $\ou_{\be}$ is bang-bang, we deduce
$\ou_{e_{n_k}}\to\ou_{\be}$ in $L^1(\Omega)$ by
Lemma~\ref{LemBBPrp}. Note that $\ou_{e_{n_k}}\in\mU_{ad}(e_{n_k})$
and the set $\bigcup_{k=1}^\infty\mU_{ad}(e_{n_k})$ is bounded in
$L^\infty(\Omega)$. Hence, we can find a constant $M>0$ such that
$\|\ou_{e_{n_k}}-\ou_{\be}\|_{L^\infty(\Omega)}\leq M$ for every
$k\in\N$. Applying H\"older's inequality we get
$$\|\ou_{e_{n_k}}-\ou_{\be}\|_{L^{p_0}(\Omega)}
  \leq\|\ou_{e_{n_k}}-\ou_{\be}\|^{1/p_0}_{L^1(\Omega)}\|\ou_{e_{n_k}}-\ou_{\be}\|^{(p_0-1)/p_0}_{L^\infty(\Omega)}
  \leq M\|\ou_{e_{n_k}}-\ou_{\be}\|^{1/p_0}_{L^1(\Omega)}\to0,$$
which verifies that $\ou_e\to\ou_{\be}$ in $L^{p_0}(\Omega)$ when
$e\to\be$ in $E$. $\hfill\Box$
\end{proof}

\begin{Corollary}\label{CorSolSbL2}
Assume that {\rm\textbf{(A1)}-\textbf{(A3)}} hold and let
$\ou_{\be}$ be a unique bang-bang solution of
problem~\eqref{BBConProb} with respect to $\be\in E$. For every
$e\in E$, let $\ou_e$ be a solution of problem~\eqref{BBConProb}.
Then, we have $\ou_e\to\ou_{\be}$ in $L^{p_0}(\Omega)$ when
$e\to\be$ in $E$.
\end{Corollary}
\begin{proof}
The proof is similar to the proof of Theorem~\ref{ThmSolSbL2}, where
the neighborhood $\oB^{p_0}_{\varepsilon}(\ou_{\be})$ is replaced by
$L^{p_0}(\Omega)$. $\hfill\Box$
\end{proof}

\medskip
We need the following lemmas that will be used in the proofs of
H\"olderian stability for solutions to problem~\eqref{BBConProb} in
the parameter $e\in E$ as well as existence of H\"olderian
selections of the solution map $S(\cdot)$.

\begin{Lemma}\label{Lem25CaEx}
Given $\we\in E$, let any $\wu\in\mU_{ad}(\we)$ be given. Then,
there exists $C_1=C_1(\wu,\we)>0$ such that
\begin{equation}\label{Cas25Ext}
    \|y_u-y_{\wu}\|_Y+\|\varphi_u-\varphi_{\wu}\|_Y\leq C_1\|u-\wu\|_{L^{p_0}(\Omega)},\
    \forall u\in L^{p_0}(\Omega),
\end{equation}
where $y_u$ and $\varphi_u$ are respectively the weak solutions of
\eqref{StateEq} and \eqref{AdjStaEq}.
\end{Lemma}
\begin{proof}
The proof is similar to the proof of \cite[Lemma~4.1]{QuiWch17}.
$\hfill\Box$
\end{proof}

\begin{Lemma}\label{LemCas27}
Given $\we\in E$, let any $\wu\in\mU_{ad}(\we)$ be given. Then, for
every $\varepsilon>0$, there exists $\rho>0$ such that for
$u\in\mU_{ad}(\we)$ with $\|u-\wu\|_{L^{p_0}(\Omega)}\leq\rho$ the
following holds
\[
\big|\mJ''_u(u,\we)v^2-\mJ''_u(\wu,\we)v^2\big|\leq\varepsilon\|z_v\|^2_{L^2(\Omega)}
\quad \forall v\in L^{p_0}(\Omega),
\]
where $z_v$ solves the linearized equation
\[
\begin{cases}
\begin{aligned}
    Az_{u,v} + \frac{\partial f}{\partial y}(x,y_{\wu+\we_y})z_{u,v} &=v\ &&\mbox{in}\ \Omega\\
                                                              z_{u,v}&=0  &&\mbox{on}\ \Gamma.
\end{aligned}
\end{cases}
\]
\end{Lemma}
\begin{proof}
The proof is similar to the proof of \cite[Lemma~2.7]{Cas12SICON}.
$\hfill\Box$
\end{proof}

\begin{Lemma}\label{lemma44}
Let $\be\in E$ and $\eta>0$ be given. Then there is a constant $K_M>0$ such that
\begin{equation}\label{EsScmJuvw}
    |\mJ''_u(u,e)(v,w)|
    \leq K_M\|z^e_{u,v}\|_{L^2(\Omega)}\|z^e_{u,w}\|_{L^2(\Omega)}
\end{equation}
holds for all $e\in B_\eta(\be)$, $u\in \mU_{ad}(e)$, $v,w\in
L^2(\Omega)$, where $z^e_{u,v}=G'(u+e_y)v$, $z^e_{u,w}=G'(u+e_y)w$
solve the linearized equation
\[
\begin{cases}
\begin{aligned}
    Az + \frac{\partial f}{\partial y}(x,y_{u+e_y})z &=v\ &&\mbox{in}\ \Omega\\
                                                   z &=0  &&\mbox{on}\ \Gamma.
\end{aligned}
\end{cases}
\]
\end{Lemma}
\begin{proof}
Let us put the set $\mU=\bigcup_{e\in\oB_\eta(\be)}\mU_{ad}(e)$. Let
us define for $e\in B_\eta(\be)$ and $u\in \mU_{ad}(e)$ the function
\[
    F(u,e)=\frac{\partial^2L}{\partial y^2}(x,y_{u+e_y})-\varphi_{u,e}\frac{\partial^2f}{\partial y^2}(x,y_{u+e_y}).
\]
Then $F$ is well-defined due to the assumptions posed on the
functions $f$ and $L$.
In addition, there is a constant $M>0$ such that
\[
\|y_{u+e_y}\|_{L^\infty(\Omega)}+\|\varphi_{u,e}\|_{L^\infty(\Omega)}\leq M
\]
holds for all $e\in B_\eta(\be)$ and $u\in \mU_{ad}(e)$. In
addition, we can find a constant $\ell_M>0$ satisfying the condition
\[
\|F(u,e)\|_{L^\infty(\Omega)}\leq \ell_M
\]
 for all $e\in B_\eta(\be)$ and $u\in \mU_{ad}(e)$.
Consequently, for every $v,w\in L^2(\Omega)$ and $u\in\mU$, it
holds that
\begin{equation}\label{EstScdJuvw}
    |J''(u+e_y)(v,w)|=\left|\int_\Omega F(u,e)z^e_{u,v}z^e_{u,w}dx\right|
    \leq\ell_M\|z^e_{u,v}\|_{L^2(\Omega)}\|z^e_{u,w}\|_{L^2(\Omega)}.
\end{equation}
Note that $G''(u+e_y)(v,w)$ is the weak solution of
\eqref{EqSolSeGvv} with respect to $y=G(u+e_y)$ and it satisfies the
condition for some constant $C\geq0$ as follows
\begin{equation}\label{PerbEsSeG}
\begin{aligned}
    \|G''(u+e_y)(v,w)\|_{L^2(\Omega)}
    &\leq C\left\|-\frac{\partial^2f}{\partial y^2}(x,y_{u+e_y})z^e_{u,v}z^e_{u,w}\right\|_{L^2(\Omega)}\\
    &\leq C\left\|\frac{\partial^2f}{\partial y^2}(x,y_{u+e_y})\right\|_{L^\infty(\Omega)}
          \|z^e_{u,v}\|_{L^2(\Omega)}\|z^e_{u,w}\|_{L^2(\Omega)}\\
    &\leq CC_{f,M}\|z^e_{u,v}\|_{L^2(\Omega)}\|z^e_{u,w}\|_{L^2(\Omega)},
\end{aligned}
\end{equation}
where $z^e_{u,v}=z_{u+e_y,v}=G'(u+e_y)v$ and
$z^e_{u,w}=z_{u+e_y,w}=G'(u+e_y)w$. From \eqref{EstScdJuvw},
\eqref{PerbEsSeG}, and \eqref{BBConProb} it follows that
$$\begin{aligned}
    |\mJ''_u(u,e)(v,w)|
    &=\Big|J''(u+e_y)(v,w)+\big(e_J,G''(u+e_y)(v,w)\big)_{L^2(\Omega)}\Big|\\
    &\leq|J''(u+e_y)(v,w)|+\|e_J\|_{L^2(\Omega)}\|G''(u+e_y)(v,w)\|_{L^2(\Omega)}\\
    &\leq\ell_M\|z^e_{u,v}\|_{L^2(\Omega)}\|z^e_{u,w}\|_{L^2(\Omega)}
     +\eta CC_{f,M}\|z^e_{u,v}\|_{L^2(\Omega)}\|z^e_{u,w}\|_{L^2(\Omega)}\\
    &=(\ell_M+\eta CC_{f,M})\|z^e_{u,v}\|_{L^2(\Omega)}\|z^e_{u,w}\|_{L^2(\Omega)},
\end{aligned}$$
which yields \eqref{EsScmJuvw} with $K_M=\ell_M+\eta CC_{f,M}$.
\hfill$\Box$
\end{proof}

\begin{Lemma}\label{lem45}
Let $\be\in E$, $\ou\in\mU_{ad}(\be)$, and $\eta>0$ be given. Then there is a constant $K_M>0$ such that
 \[
  \|\varphi_{u,\be}-\varphi_{\ou,\be}\|_{L^\infty(\Omega)} \le K_M \|u-\ou\|_{L^1(\Omega)}
 \]
holds for all $e\in B_\eta(\be)$, $u\in \mU_{ad}(e)$.
\end{Lemma}
\begin{proof}
The proof is similar to the proof of
\cite[Lemma~2.6]{CaWaWa17SICON}. \qed
\end{proof}

\begin{Lemma}\label{LemEsFDrMJ}
Let $\be\in E$, $\ou\in\mU_{ad}(\be)$, and $\eta>0$ be given. Then
there is a constant $K_M>0$ such that
 \[
  \|\varphi_{\ou,e}-\varphi_{\ou,\be}\|_{L^\infty(\Omega)} \le K_M \|e-\be\|_E
 \]
holds for all $e\in B_\eta(\be)$.
\end{Lemma}
\begin{proof}
Since $\varphi_{\ou,e}$ and $\varphi_{\ou,\be}$ are the weak
solutions of \eqref{PertbAdjEq} with respect to $e$ and $\be$. Thus,
we have
$$\begin{cases}
\begin{aligned}
    &A^*(\varphi_{\ou,e}-\varphi_{\ou,\be})+
     \dfrac{\partial f}{\partial y}(x,y_{\ou+e_y})(\varphi_{\ou,e}-\varphi_{\ou,\be})
     =\dfrac{\partial L}{\partial y}(x,y_{\ou+e_y})-\dfrac{\partial L}{\partial y}(x,y_{\ou+\be_y})\\
    &\hspace{4.8cm}-\left(\dfrac{\partial f}{\partial y}(x,y_{\ou+e_y})
                    -\dfrac{\partial f}{\partial y}(x,y_{\ou+\be_y})\right)\varphi_{\ou,\be}
                    +e_J-\be_J\ &&\mbox{in}\ \Omega\\
    &\hspace{6cm}\varphi_{\ou,e}-\varphi_{\ou,\be}=0               &&\mbox{on}\ \Gamma,
\end{aligned}
\end{cases}$$
By our assumptions, there exist $C>0$, $\ell_1>0$, $\ell_2>0$ such
that
$$\begin{aligned}
    \|\varphi_{\ou,e}-\varphi_{\ou,\be}\|_{L^\infty(\Omega)}
    &\leq C\left(\left\|\frac{\partial L}{\partial y}(\cdot,y_{\ou+e_y})
     -\frac{\partial L}{\partial y}(\cdot,y_{\ou+\be_y})\right\|_{L^2(\Omega)}\right.\\
    &\quad+\left.\left\|\dfrac{\partial f}{\partial y}(\cdot,y_{\ou+e_y})
     -\dfrac{\partial f}{\partial y}(\cdot,y_{\ou+\be_y})\right\|_{L^2(\Omega)}\|\varphi_{\ou,\be}\|_{L^2(\Omega)}
     +\|e_J-\be_J\|_{L^2(\Omega)}\right)\\
    &\leq C\big(\ell_1\|y_{\ou+e_y}-y_{\ou+\be_y}\|_{L^2(\Omega)}+\|e_J-\be_J\|_{L^2(\Omega)}\big)\\
    &\leq C\big(\ell_1\ell_2\|e_y-\be_y\|_{L^2(\Omega)}+\|e_J-\be_J\|_{L^2(\Omega)}\big)\\
    &\leq K_M\|e-\be\|_E,
\end{aligned}$$
where $K_M:=C\max\{\ell_1\ell_2,1\}$. \qed
\end{proof}

\begin{Lemma}\label{lem48}
Let $\ou\in\mU_{ad}(\be)$ and $\rho>0$ be given. Then, there exists
$c>0$ such that for all $\hu\in L^\infty(\Omega)$ with
$\|\hu-\ou\|_{L^2(\Omega)}<\rho$ it holds
\[
 \|z_{\hu,w}^\be\|_{L^2(\Omega)} \le c \|w\|_{L^1(\Omega)},~\forall w\in L^2(\Omega),
\]
and
\[
 \|z_{\ou,v}^\be-z_{\hu,v}^\be\|_{L^2(\Omega)}\leq c\|z_{\ou,v}^\be\|_{L^2(\Omega)}\|\hu-\ou_{\be}\|_{L^2(\Omega)},
\]
where we define $z_{\hu,w}^\be:=z_{\hu+\be_y,w}=G'(\hu+\be_y)w$, and
similarly for $z_{\ou,v}^\be$ and $z_{\hu,v}^\be$.
\end{Lemma}
\begin{proof}
It follows from \cite[Lemma 4.2]{QuiWch17}. $\hfill\Box$
\end{proof}

\medskip
We also need the following extension of Proposition~\ref{PropFstCd}
to perturbed feasible sets.

\begin{Lemma}\label{LemEstFrmJ}
Assume that {\rm\textbf{(A4)}} holds at $\ou_{\be}$. Assume further
that there is $\sigma>0$ such that
\[
 \beta + \be_\beta - (\alpha+ \be_\alpha) \ge \sigma \text{ a.e.\ on } \Omega.
\]
Take $0<\eta<\sigma/4$. Let $e\in B_\eta(\be)$ and let
$\ou_e\in\mU_{ad}(e)$ satisfy the first-order optimality system
\eqref{PertbStateEq}-\eqref{PertbVarIneq}. Then there are $c>0$ and
$\kappa'>0$ independent of $e$ and $u_e\in \mU_{ad}(\be)$ such that
$$\begin{aligned}
    &\big(\mJ'_u(\ou_e,e) - \mJ'_u(\ou_\be,\be)\big)(\ou_\be-\ou_e)\\
    &\qquad\ge\kappa' \|\ou_e-\ou_\be \|_{L^1(\Omega)}^{1+\frac{1}{\ae}}+ \frac12 \mJ'_u(\ou_\be,\be)(u_e-\ou_\be)\\
    &\qquad\quad- c \left( \|e-\be\|_E^{\frac{1}{\ae}}
    + \|\varphi_{\ou_e,e}-\varphi_{\ou_\be,\be}\|_{L^\infty(\Omega)}
    + \|\ou_e-\ou_\be\|_{L^1(\Omega)}\right) \|e-\be\|_E,
\end{aligned}$$
and $\|\ou_e-u_e\|_{L^\infty(\Omega)}\le \|e-\be\|_E$.
\end{Lemma}
\begin{proof}
Due to the perturbation in the feasible set, we have
$\ou_\be\not\in\mU_{ad}(e)$ and $\ou_e\not\in \mU_{ad}(\be)$ in
general. First, we construct controls $u_e \approx \ou_e$ with
$u_e\in \mU_{ad}(\be)$ and $u_\be \approx \ou_\be$ with $u_\be\in
\mU_{ad}(e)$. Let us define
 \[
  \Omega_\sigma:=\left\{x\in\Omega\Bst |\ou_e-\ou_\be| <\frac\sigma2\right\}.
 \]
 Then on $\Omega_\sigma$ we have the implications
 \[
  \ou_\be = \alpha + \be_\alpha \ \Longrightarrow \ u_e \le \beta + e_\beta - \frac\sigma4,
 \]
 and
 \[
  \ou_\be = \beta +\be_\beta \ \Longrightarrow \ u_e \ge \alpha + e_\alpha + \frac\sigma4.
 \]
Let us define
\[
 u_e:= \chi_{ \Omega_\sigma \cap \Omega_1(\ou_\be,\be)}  ( \ou_e - (e_\alpha-\be_\alpha))
  + \chi_{ \Omega_\sigma \cap \Omega_3(\ou_\be,\be)}  ( \ou_e - (e_\beta-\be_\beta))
  + \chi_{\Omega \setminus \Omega_\sigma} \proj_{\mU_{ad}(\be)}(\ou_e),
\]
and
\[
 u_\be := \chi_{ \Omega_\sigma \cap \Omega_1(\ou_\be,\be)}  ( \ou_\be + (e_\alpha-\be_\alpha))
  + \chi_{ \Omega_\sigma \cap \Omega_3(\ou_\be,\be)}  ( \ou_\be + (e_\beta-\be_\beta))
  + \chi_{\Omega \setminus \Omega_\sigma} \proj_{\mU_{ad}(e)}(\ou_\be).
\]
Due to the definition of $\Omega_\sigma$ it holds $u_e\in \mU_{ad}(\be)$ and $u_\be\in \mU_{ad}(e)$.
In addition, we have the important relation
\[
 u_e - \ou_e = - (u_\be - \ou_\be)  \text{ on } \Omega_\sigma.
\]
Since the projection is Lipschitz continuous with respect to changes of upper and lower bounds,
we have
\[
 \|u_e - \ou_e\|_{L^\infty(\Omega)} , \|u_\be - \ou_\be\|_{L^\infty(\Omega)} \le \|e-\be\|_E.
\]
Using these feasible approximations, by Proposition \ref{PropFstCd}
and \eqref{PertbVarIneq}--\eqref{PerbDermJ} we get
\begin{multline}\label{eq463}
\begin{aligned}
    &\quad\,\big(\mJ'_u(\ou_e,e) - \mJ'_u(\ou_\be,\be)\big)(\ou_\be-\ou_e)\\
    &= \mJ'_u(\ou_e,e)(\ou_\be-u_\be+u_\be-\ou_e)  - \mJ'_u(\ou_\be,\be)(\ou_\be-u_e+u_e-\ou_e) \\
    &\ge\mJ'_u(\ou_e,e)(\ou_\be-u_\be) - \mJ'_u(\ou_\be,\be)(u_e-\ou_e)
     + \frac\kappa2 \|u_e-\ou_\be\|_{L^1(\Omega)}^{1+\frac1\ae} + \frac12 \mJ'_u(\ou_\be,\be)(u_e-\ou_\be).
\end{aligned}
\end{multline}
We can rewrite
\begin{multline*}
\begin{aligned}
    &\quad\;\mJ'_u(\ou_e,e)(\ou_\be-u_\be) - \mJ'_u(\ou_\be,\be)(u_e-\ou_e)\\
    &=\big(\mJ'_u(\ou_e,e) - \mJ'_u(\ou_\be,\be)\big)d
     + \mJ'_u(\ou_e,e)( \chi_{\Omega \setminus \Omega_\sigma} (\ou_\be-u_\be))
     - \mJ'_u(\ou_\be,\be)( \chi_{\Omega \setminus \Omega_\sigma} (u_e-\ou_e)),
\end{aligned}
\end{multline*}
where
\[
 d:= \chi_{\Omega_\sigma}(\ou_\be-u_\be)=\chi_{\Omega_\sigma}(u_e-\ou_e)
\]
satisfying $\|d\|_{L^\infty(\Omega) } \le \|e-\be\|_E$. Due to
Tchebyshev's inequality, the measure of
$\Omega\setminus\Omega_\sigma$ is bounded by $2\sigma^{-1}
\|\ou_e-\ou_\be\|_{L^1(\Omega)}$. Then we can estimate with $c>0$
independent of $e$ as follows
$$\begin{aligned}
    &\big|\mJ'_u(\ou_e,e)(\ou_\be-u_\be) - \mJ'_u(\ou_\be,\be)(u_e-\ou_e)\big|\\
    &\qquad\le |\Omega|\cdot \|\varphi_{\ou_e,e}-\varphi_{\ou_\be,\be}\|_{L^\infty(\Omega)} \|e-\be\|_E \\
    &\qquad\quad +\big(\|\varphi_{\ou_e,e}-\varphi_{\ou_\be,\be}\|_{L^\infty(\Omega)}
     +2\|\varphi_{\ou_\be,\be}\|_{L^\infty(\Omega)}\big)2\sigma^{-1}\|\ou_e-\ou_\be\|_{L^1(\Omega)}\|e-\be\|_E\\
    &\qquad\le c\Big( \|\varphi_{\ou_e,e}-\varphi_{\ou_\be,\be}\|_{L^\infty(\Omega)}
     +\big(1+\|\varphi_{\ou_e,e}-\varphi_{\ou_\be,\be}\|_{L^\infty(\Omega)}\big)\|\ou_e-\ou_\be\|_{L^1(\Omega)}\Big)
     \|e-\be\|_E.
\end{aligned}$$
Since $\|\ou_e-\ou_\be\|_{L^1(\Omega)}$ is uniformly bounded with
respect to $e$ due to the presence of the control constraints, we
can simplify this inequality to
\begin{equation}\label{eq464}
\begin{aligned}
    &|\mJ'_u(\ou_e,e)(\ou_e-u_e) - \mJ'_u(\ou_\be,\be)(\ou_e-u_\be)|\\
    &\qquad\le c \left( \|\varphi_{\ou_e,e}-\varphi_{\ou_\be,\be}\|_{L^\infty(\Omega)}
    +\|\ou_e-\ou_\be\|_{L^1(\Omega)}\right) \|e-\be\|_E.
\end{aligned}
\end{equation}
It remains to develop a lower bound of $\|u_e-\ou_\be\|_{L^1(\Omega)}^{1+\frac1\ae}$.
By construction of $u_e$, we have
\[
 \|\ou_e-\ou_\be \|_{L^1(\Omega)} \le \|u_e-\ou_\be\|_{L^1(\Omega)} + |\Omega|\cdot \|e-\be\|_E.
\]
This implies
\begin{equation}\label{eq465}
  \|\ou_e-\ou_\be\|_{L^1(\Omega)}^{1+\frac{1}{\ae}}
  \le 2^{\frac1\ae}\left(\|u_e-\ou_\be\|_{L^1(\Omega)}^{1+\frac1\ae}
  + |\Omega|^{1+\frac1\ae}\|e-\be\|_E^{1+\frac1\ae}\right).
\end{equation}
Collecting the inequalities \eqref{eq463}--\eqref{eq465} yields
$$\begin{aligned}
    &\big(\mJ'_u(\ou_e,e) - \mJ'_u(\ou_\be,\be)\big)(\ou_\be-\ou_e)\\
    &\qquad\ge\kappa' \|\ou_e-\ou_\be \|_{L^1(\Omega)}^{1+\frac{1}{\ae}}+ \frac12 \mJ'_u(\ou_\be,\be)(u_e-\ou_\be)\\
    &\qquad\quad-c\left( \|e-\be\|_E^{\frac{1}{\ae}}
     + \|\varphi_{\ou_e,e}-\varphi_{\ou_\be,\be}\|_{L^\infty(\Omega)}
     + \|\ou_e-\ou_\be\|_{L^1(\Omega)}\right) \|e-\be\|_E
\end{aligned}$$
with $\kappa':= \kappa 2^{1-\frac1\ae}$ and $c$ independent of $e,\ou_e$.
$\hfill\Box$
\end{proof}

\begin{Theorem}\label{ThmStabKKT}
Let $\ou_{\be}$ be a strict bang-bang solution of
problem~\eqref{BBConProb} for $\be\in E$ and assume that
{\rm\textbf{(A1)}-\textbf{(A4)}} hold. Assume that the second-order
condition \eqref{SOrdCd} holds at $\ou_\be$. Then, there exist
$\eta>0$ and $c>0$ such that
\begin{equation}\label{MnHldrEst}
    \|\ou_e-\ou_\be\|_{L^1(\Omega)}\le c\|e-\be\|_E^{\min\{\ae,1\}}
\end{equation}
for all $e\in B_\eta(\be)$ and for any $\ou_e\in\mU_{ad}(e)\cap
B^{p_0}_\eta(\ou_\be)$ satisfying the first-order optimality system
\eqref{PertbStateEq}-\eqref{PertbVarIneq}.
\end{Theorem}
\begin{proof}
By Lemma~\ref{LemEstFrmJ}, we have
$$\begin{aligned}
    &\big(\mJ'_u(\ou_e,e)-\mJ'_u(\ou_\be,\be)\big)(\ou_\be-\ou_e)\\
    &\qquad\ge \kappa' \|\ou_e-\ou_\be\|_{L^1(\Omega)}^{1+\frac{1}{\ae}}
     +\frac12 \mJ'_u(\ou_\be,\be)(u_e-\ou_\be)\\
    &\qquad\quad-c\left(\|e-\be\|_E^{\frac{1}{\ae}}
     +\|\varphi_{\ou_e,e}-\varphi_{\ou_\be,\be}\|_{L^\infty(\Omega)}
     +\|\ou_e-\ou_\be\|_{L^1(\Omega)}\right)\|e-\be\|_E
\end{aligned}$$
with $u_e\in \mU_{ad}(\be)$ and $\|\ou_e-u_e\|_{L^\infty(\Omega)}\le \|e-\be\|_E$  as above. We write
$$\begin{aligned}
    &\quad\ \big(\mJ'_u(\ou_e,e)-\mJ'_u(\ou_\be,\be)\big)(\ou_\be-\ou_e)\\
    &=\big(\mJ'_u(\ou_e,e) - \mJ'_u(\ou_e,\be)\big)(\ou_\be-\ou_e)
     +\big(\mJ'_u(\ou_e,\be) - \mJ'_u(\ou_\be,\be)\big)(\ou_\be-\ou_e).
\end{aligned}$$
Using the representation \eqref{PerbDermJ} of $\mJ'_u$ by adjoint states and the estimate of Lemma \ref{LemEsFDrMJ},
we obtain
$$\begin{aligned}
    \big|\big(\mJ'_u(\ou_e,e) - \mJ'_u(\ou_e,\be)\big)(\ou_\be-\ou_e)\big|
    &\le\|\mJ'_u(\ou_e,e)-\mJ'_u(\ou_e,\be)\|_{L^\infty(\Omega)}\|\ou_\be-\ou_e\|_{L^1(\Omega)}\\
    & = \|\varphi_{\ou_e,e}-\varphi_{\ou_e,\be}\|_{L^\infty(\Omega)}\|\ou_\be-\ou_e\|_{L^1(\Omega)}\\
    &\le c\|e-\be\|_E\|\ou_\be-\ou_e\|_{L^1(\Omega)}.
\end{aligned}$$
In addition,  by Lemmas \ref{lem45} and \ref{LemEsFDrMJ} we get
$$\|\varphi_{\ou_e,e}-\varphi_{\ou_\be,\be}\|_{L^\infty(\Omega)}
\le c(\|e-\be\|_E+\|\ou_e-\ou_\be\|_{L^1(\Omega)}).$$
This shows
\begin{equation}\label{eq416}
\begin{aligned}
    &\kappa'\|\ou_e-\ou_\be\|_{L^1(\Omega)}^{1+\frac{1}{\ae}}
     + (\mJ'_u(\ou_e,\be) - \mJ'_u(\ou_\be,\be))(\ou_e-\ou_\be)
     +\frac12 \mJ'_u(\ou_\be,\be)(u_e-\ou_\be)\\
    &\qquad\quad\;\le c \left( \|e-\be\|_E^{\frac{1}{\ae}}
    +\|e-\be\|_E+\|\ou_e-\ou_\be\|_{L^1(\Omega)}\right)\|e-\be\|_E.
\end{aligned}
\end{equation}
By Taylor expansion, we find
$$\big(\mJ'_u(\ou_e,\be)-\mJ'_u(\ou_\be,\be)\big)(\ou_e-\ou_\be)=\mJ''_u(\hu,\be)(\ou_e-\ou_{\be})^2,$$
where $\hu=\ou_{\be}+\theta(\ou_e-\ou_{\be})$ and $\theta\in(0,1)$.
Let us define
$$\Omega_\tau:=\{x\in \Omega:|\varphi_{\ou_{\be}}|\le\tau\}.$$
We now define
\[
 v = \chi_{\Omega_\tau}(u_e - \ou_\be), \quad
 w = \chi_{\Omega\setminus\Omega_\tau} ( u_e -  \ou_\be), \quad
 \widetilde w := \ou_e-u_e,
\]
such that $v+w+\widetilde w= \ou_e - \ou_\be$ and $v\in
C^\tau_{\ou_{\be},p_0}$, for the definition of
 $C^\tau_{\ou_{\be},p_0}$ see \eqref{ExCriConce}.
Moreover, we have $\|\widetilde w\|_{L^\infty(\Omega)}\le
\|e-\be\|_E$. Due to the feasibility $u_e\in\mU_{ad}(\be)$, we have
\[
\mJ'_u(\ou_\be,\be)(u_e-\ou_\be)=\int_\Omega |\varphi_{\ou_{\be}}|\cdot |u_e-\ou_{\be}|dx\geq \tau \|w\|_{L^1(\Omega)}.
\]
From the definition of $v$ and $w$ we get
$$\begin{aligned}
    &\mJ''_u(\hu,\be)(\ou_e-\ou_{\be})^2+\frac12\mJ'_u(\ou_\be,\be)(u_e-\ou_\be)\\
    &\qquad\ge\frac\tau2\|w\|_{L^1(\Omega)}+\mJ''_u(\ou_\be,\be)v^2+(\mJ''_u(\hu,\be)-\mJ''_u(\ou_\be,\be))v^2\\
    &\qquad\quad+\mJ''_u(\hu,\be)(w+ \widetilde w)^2 + 2 \mJ''_u(\hu,\be)(v, w+ \widetilde w).
\end{aligned}$$
Let us set for abbreviation $z_v:=z_{\ou_\be,v}^\be$,
$z_{\hu,v}:=z_{\hu,v}^\be$, and similarly for $z_{\hu,w}$,
$z_{\hu,\widetilde w}$. Using the second-order condition
\eqref{SOrdCd}, the continuity estimate of $\mJ_u''$ of Lemma
\ref{LemCas27}, and the estimate of $\mJ_u''$ of Lemma
\ref{lemma44}, we get
\begin{equation}\label{eq417}
\begin{aligned}
    &\mJ''_u(\hu,\be)(\ou_e-\ou_{\be})^2+\frac12\mJ'_u(\ou_\be,\be)(u_e-\ou_\be)\\
    &\qquad\ge\frac\tau2\|w\|_{L^1(\Omega)}+\delta\|z_v\|_{L^2(\Omega)}^2-\frac\delta4\|z_v\|_{L^2(\Omega)}^2
     -2K_M\big(\|z_{\hu,w}\|_{L^2(\Omega)}^2+\|z_{\hu,\widetilde w}\|_{L^2(\Omega)}^2\big)\\
    &\qquad\quad-2K_M\big(\|z_v\|_{L^2(\Omega)}+\|z_v-z_{\hu,v}\|_{L^2(\Omega)}\big)\big(\|z_{\hu,w}\|_{L^2(\Omega)}
     +\|z_{\hu,\widetilde w}\|_{L^2(\Omega)}\big)
    \end{aligned}
\end{equation}
for all $\ou_e$ in some ball $B^{p_0}_\eta(\ou_\be)$ with $\eta>0$.
Using Lemma~\ref{lem48}, we estimate
\[
 \|z_{\hu,w}\|_{L^2(\Omega)}\leq c\|w\|_{L^1(\Omega)}
\]
and
\[
  \|z_{\hu,\widetilde w}\|_{L^2(\Omega)}\leq c\|\widetilde w\|_{L^\infty(\Omega)} \le c \|e-\be\|_E.
\]
Applying Lemma~\ref{lem48}, we find
\[
 \|z_v-z_{\hu,v}\|_{L^2(\Omega)} \le c \eta \|z_v\|_{L^2(\Omega)}.
\]
Using this estimate in \eqref{eq417}, we obtain
$$\begin{aligned}
    &\mJ''_u(\hu,\be)(\ou_e-\ou_{\be})^2+\frac12\mJ'_u(\ou_\be,\be)(u_e-\ou_\be)\\
    &\qquad\ge\frac\tau2\|w\|_{L^1(\Omega)} + \frac34\delta\|z_v\|_{L^2(\Omega)}^2 \\
    &\qquad\quad -c\left( \|w\|_{L^1(\Omega)}^2
     +\|e-\be\|_E^2 + \|z_v\|_{L^2(\Omega)}(\|w\|_{L^1(\Omega)} + \|e-\be\|_E)\right)
\end{aligned}$$
with some $c>0$ independent of $e$ and $\ou_e$. Using Young's
inequality, the following inequality can be derived:
$$\begin{aligned}
    &\mJ''_u(\hu,\be)(\ou_e-\ou_{\be})^2+\frac12\mJ'_u(\ou_\be,\be)(u_e-\ou_\be) \\
    &\qquad\ge\|w\|_{L^1(\Omega)} \left( \frac\tau2 - c_1 \|w\|_{L^1(\Omega)} \right)
     + \frac12\delta\|z_v\|_{L^2(\Omega)}^2 - c_2\|e-\be\|_E^2.
\end{aligned}$$
By making $\eta$ smaller if necessary, we can achieve
\[
 \|w\|_{L^1(\Omega)} \left( \frac\tau2 - c_1 \|w\|_{L^1(\Omega)} \right)  \ge \frac\tau4 \|w\|_{L^1(\Omega)}.
\]
This shows
$$\mJ''_u(\hu,\be)(\ou_e-\ou_{\be})^2+\frac12\mJ'_u(\ou_\be,\be)(u_e-\ou_\be)
  \ge\frac\tau4\|w\|_{L^1(\Omega)}+\frac\delta2\|z_v\|_{L^2(\Omega)}^2-C\|e-\be\|_E^2.$$
Together with \eqref{eq416}, this implies
$$\begin{aligned}
    &\frac\tau4\|w\|_{L^1(\Omega)}+\frac\delta2\|z_v\|_{L^2(\Omega)}^2
     +\kappa'\|\ou_e-\ou_\be\|_{L^1(\Omega)}^{1+\frac{1}{\ae}}\\
    &\qquad\le c\left(\|e-\be\|_E^{\frac{1}{\ae}}+\|e-\be\|_E
     +\|\ou_e-\ou_\be\|_{L^1(\Omega)}\right) \|e-\be\|_E.
\end{aligned}$$
With Young's inequality we obtain
\[
 c \|\ou_e-\ou_\be\|_{L^1(\Omega)} \|e-\be\|_E \le \frac{\kappa'}2\|\ou_e-\ou_\be\|_{L^1(\Omega)}^{1+\frac{1}{\ae}}
 + c \|e-\be\|_E^{\ae+1}.
\]
Thus, we arrive at the inequality
\[
 \frac\tau4\|w\|_{L^1(\Omega)}+\frac\delta2\|z_v\|_{L^2(\Omega)}^2
     +\kappa'\|\ou_e-\ou_\be\|_{L^1(\Omega)}^{1+\frac{1}{\ae}}
     \le c\left( \|e-\be\|_E^{1+\frac{1}{\ae}}+\|e-\be\|_E^2 + \|e-\be\|_E^{\ae+1}\right).
\]
For $\ae\in[0,1]$, $\ae+1$ is the smallest exponent on the right-hand side,
if $\ae>1$, then $1+\frac{1}{\ae}$ is the smallest exponent.
This proves
$$\|\ou_e-\ou_\be\|_{L^1(\Omega)}\le C\|e-\be\|_E^{\min\{\ae,1\}},$$
and therefore we obtain
\eqref{MnHldrEst}. \qed
\end{proof}

\medskip
The following theorem shows that the (global) solution map
$S:\dom\mU_{ad}\rightrightarrows L^1(\Omega)$ admits a local upper
H\"{o}lderian selection at a given point $(\be,\ou_{\be})\in\gph S$
provided that for every $e\in\dom\mU_{ad}$ near $\be$,
problem~\eqref{BBConProb} has a (global) solution $\ou_e$ near
$\ou_{\be}$.

\begin{Theorem}\label{ThmLUpHldr}
Assume that all the assumptions of Theorem~\ref{ThmStabKKT} are
satisfied and let $(\be,\ou_{\be})\in\gph S$ be such that
$\ou_{\be}$ is strict in a neighborhood
$\oB^{p_0}_{\varepsilon}(\ou_{\be})$ with $\varepsilon>0$. Assume
further that for every $e\in\dom\mU_{ad}$ near $\be$,
problem~\eqref{BBConProb} has a solution $\ou_e$ satisfying
$\ou_e\in\oB^{p_0}_{\varepsilon}(\ou_{\be})$. Then, the solution map
$S:\dom\mU_{ad}\rightrightarrows L^1(\Omega)$ admits a local upper
H\"{o}lderian selection with the exponent $\ae\in[0,1]$ at the point
$(\be,\ou_{\be})$.
\end{Theorem}
\begin{proof}
According to Theorems~\ref{ThmSolSbL2} and \ref{ThmStabKKT}, there
exist constants $\eta>0$ and $c>0$ satisfying
\begin{equation}\label{EstHolLip}
    \|\ou_e-\ou_{\be}\|_{L^1(\Omega)}\leq c\|e-\be\|_E^{\min\{\ae,1\}},\
    \forall e\in\oB_\eta(\be),
\end{equation}
where $\ou_e$ is a solution of problem~\eqref{BBConProb} with
respect to $e\in E$ satisfying
$\ou_e\in\oB^{p_0}_{\varepsilon}(\ou_{\be})$. Define a single-valued
function $h:\dom\mG\to L^1(\Omega)$ by $h(\be)=\ou_{\be}$ and
$h(e)=\ou_e$ for $e\in\dom\mG$. Then, for
$e\in\oB_\eta(\be)\cap\dom\mG$, by \eqref{EstHolLip} we obtain
$$\|h(e)-h(\be)\|_{L^1(\Omega)}\leq c\|e-\be\|^{\ae'}_E,$$
which yields that $h$ is a local upper H\"{o}lderian selection of
$S(\cdot)$ at the point $(\be,\ou_{\be})$, where the exponent
$\ae'=\min\{\ae,1\}\in[0,1]$. $\hfill\Box$
\end{proof}

\begin{Corollary}\label{CorLoUpLpSe}
Assume that all the assumptions of Theorem~\ref{ThmStabKKT} are
satisfied and let $(\be,\ou_{\be})\in\gph S$ be such that
$S(\be)=\{\ou_{\be}\}$. Then, the solution map
$S:\dom\mU_{ad}\rightrightarrows L^1(\Omega)$ of
problem~\eqref{BBConProb} has a local upper H\"{o}lderian selection
at $(\be,\ou_{\be})$ with the exponent $\ae\in[0,1]$.
\end{Corollary}
\begin{proof}
By applying Theorem~\ref{ThmStabKKT} and Corollary~\ref{CorSolSbL2}
and arguing similarly as in the proof of Theorem~\ref{ThmLUpHldr},
we obtain the assertion of the corollary. $\hfill\Box$
\end{proof}

\subsection{Lower estimate for regular subdifferential of $\mu(\cdot)$}

In this subsection, we will establish a characterization of regular
subgradients of the marginal function $\mu(\cdot)$ in a subspace
$E^*_1$ of $E^*$, where $E^*_1$ is defined as follows
$$\begin{aligned}
    E^*_1
    &:=L^2(\Omega)\times L^2(\Omega)\times L^1(\Omega)\times L^1(\Omega)\\
    &~\subset L^2(\Omega)\times L^2(\Omega)\times L^\infty(\Omega)^*\times L^\infty(\Omega)^*=E^*.
\end{aligned}$$
We now define the set
$$\begin{aligned}
    \widehat{\Xi}\big((\be,\ou_{\be});\gph\mU_{ad}\big)
    &=\Big\{(e^*,u^*)\in E^*_1\times L^2(\Omega)\Bst e^*=(0,0,e^*_\alpha,e^*_\beta),\,u^*=-e^*_\alpha-e^*_\beta,\\
    &\qquad e^*_\alpha|_{\Omega_1(\be,\ou_{\be})}\geq0,\,
            e^*_\alpha|_{\Omega\setminus\Omega_1(\be,\ou_{\be})}=0,\\
    &\qquad e^*_\beta|_{\Omega_3(\be,\ou_{\be})}\leq0,\,e^*_\beta|_{\Omega\setminus\Omega_3(\be,\ou_{\be})}=0\Big\}.
\end{aligned}$$
By arguing similarly as the proof of Lemma~\ref{LmFNCGpUad} we get
$$\widehat{\Xi}\big((\be,\ou_{\be});\gph\mU_{ad}\big)\subset\widehat{N}\big((\be,\ou_{\be});\gph\mU_{ad}\big).$$
Consequently, by setting
\begin{equation}\label{CmptLmdMG}
\begin{aligned}
    \widehat{\Lambda}^*\mG(\be,\ou_{\be})(u^*)
    &=\big\{e^*\in E^*_1\bst (e^*,-u^*)\in\widehat{\Xi}\big((\be,\ou_{\be});\gph\mG\big)\big\}\\
    &=\Big\{e^*\in E^*_1\Bst e^*=(0,0,e^*_\alpha,e^*_\beta),\,u^*=e^*_\alpha+e^*_\beta,\\
    &\qquad e^*_\alpha|_{\Omega_1(\be,\ou_{\be})}\geq0,\,e^*_\alpha|_{\Omega\setminus\Omega_1(\be,\ou_{\be})}=0,\\
    &\qquad e^*_\beta|_{\Omega_3(\be,\ou_{\be})}\leq0,\,e^*_\beta|_{\Omega\setminus\Omega_3(\be,\ou_{\be})}=0\Big\},
\end{aligned}
\end{equation}
we deduce that
\begin{equation}\label{EsLmdaMG}
    \widehat{\Lambda}^*\mG(\be,\ou_{\be})(u^*)\subset\widehat{D}^*\mG(\be,\ou_{\be})(u^*).
\end{equation}
Motivated by the estimate \eqref{EsLmdaMG}, we are going to
establish a lower estimate for $\widehat{\partial}\mu(\be)$ via a
characterization of regular subgradients of the marginal function
$\mu(\cdot)$ in the subspace $E^*_1$ of $E^*$ in the forthcoming
theorem.

\begin{Theorem}\label{ThmLwEsFrSb}
Assume that {\rm\textbf{(A1)}-\textbf{(A3)}} hold and let
$(\be,\ou_{\be})\in\gph S$ be given. Then, for any
$\he^*=(\he^*_y,\he^*_J,\he^*_\alpha,\he^*_\beta)\in\widehat{\partial}\mu(\be)\cap
E^*_1$, the following holds
\begin{equation}\label{DKFrSbDiff2}
\begin{cases}
   \he^*_y=\varphi_{\ou_{\be},\be},\\
   \he^*_J=y_{\ou_{\be}+\be_y},\\
   \he^*_\alpha|_{\Omega_1(\be,\ou_{\be})}\geq0,\,
   \he^*_\alpha|_{\Omega\setminus\Omega_1(\be,\ou_{\be})}=0,\\
   \he^*_\beta|_{\Omega_3(\be,\ou_{\be})}\leq0,\,\he^*_\beta|_{\Omega\setminus\Omega_3(\be,\ou_{\be})}=0,\\
   \he^*_\alpha+\he^*_\beta=\varphi_{\ou_{\be},\be}.
\end{cases}
\end{equation}
In addition, assume that the solution map $S(\cdot)$ admits a local
upper H\"{o}lderian selection $h(\cdot)$ with $h(\be)=\ou_{\be}$,
$h(e)=\ou_e$, and
\begin{equation}\label{HldrSelct}
    \|h(e)-h(\be)\|_{L^1(\Omega)}\leq c\|e-\be\|^{\ae}_E,~\forall e\in\oB_\eta(\be)\cap\dom\mG,
\end{equation}
for some $\eta>0$, $c>0$, and $\ae>1/2$. If an element
$\he^*=(\he^*_y,\he^*_J,\he^*_\alpha,\he^*_\beta)\in E^*_1$
satisfies \eqref{DKFrSbDiff2}, then
$\he^*\in\widehat{\partial}\mu(\be)$.
\end{Theorem}
\begin{proof}
By definition, we have
$\he^*=(\he^*_y,\he^*_J,\he^*_\alpha,\he^*_\beta)\in\widehat{\partial}\mu(\be)$
if and only if
\begin{equation}\label{DefSubMU}
    \liminf_{e\to\be}\frac{\mu(e)-\mu(\be)-\langle\he^*,e-\be\rangle}{\|e-\be\|_E}\geq0.
\end{equation}

Suppose that
$\he^*=(\he^*_y,\he^*_J,\he^*_\alpha,\he^*_\beta)\in\widehat{\partial}\mu(\be)\cap
E^*_1$. We verify that
$\he^*=(\he^*_y,\he^*_J,\he^*_\alpha,\he^*_\beta)$ satisfies
\eqref{DKFrSbDiff2}. By the inclusion
$\he^*\in\widehat{\partial}\mu(\be)$, we have \eqref{DefSubMU}.

Let
$e=(\be_y,\be_J,e_\alpha,e_\beta)\to(\be_y,\be_J,\be_\alpha,\be_\beta)=\be$
with $(e_\alpha,e_\beta)$ being chosen the same in the proof of
Lemma~\ref{LmFNCGpUad}. Then, we have $\ou_{\be}\in\mG(e)$, and thus
$\mJ(\ou_{\be},e)\geq\mu(e)$. Since $\mJ(\cdot,\cdot)$ does not
depend on $e_\alpha$ and $e_\beta$, we have
$\mJ(\ou_{\be},\be)=\mJ(\ou_{\be},e)$ which yields
$\mJ(\ou_{\be},\be)\geq\mu(e)$. Note that
$\mu(\be)=\mJ(\ou_{\be},\be)$ due to $(\be,\ou_{\be})\in\gph S$.
Consequently, from \eqref{DefSubMU} we obtain
\begin{equation}\label{LmifEalEbe}
    \liminf_{e\to\be}
    \frac{-\langle\he^*_\alpha,e_\alpha-\be_\alpha\rangle-\langle\he^*_\beta,e_\beta-\be_\beta\rangle}
         {\|e_\alpha-\be_\alpha\|_{L^\infty(\Omega)}+\|e_\beta-\be_\beta\|_{L^\infty(\Omega)}}\geq0.
\end{equation}
Using \eqref{LmifEalEbe} and arguing similarly as the proof of
Lemma~\ref{LmFNCGpUad} we deduce that
\begin{equation}\label{NcsCnd11}
\begin{cases}
   \he^*_\alpha|_{\Omega_1(\be,\ou_{\be})}\geq0,\,
   \he^*_\alpha|_{\Omega\setminus\Omega_1(\be,\ou_{\be})}=0,\\
   \he^*_\beta|_{\Omega_3(\be,\ou_{\be})}\leq0,\,\he^*_\beta|_{\Omega\setminus\Omega_3(\be,\ou_{\be})}=0.
\end{cases}
\end{equation}

Let
$e=(e_y,e_J,\be_\alpha,\be_\beta)\to(\be_y,\be_J,\be_\alpha,\be_\beta)=\be$.
Then, $\mJ(\ou_{\be},e)\geq\mu(e)$ by $\ou_{\be}\in\mG(\be)=\mG(e)$.
Combining this with \eqref{DefSubMU} we obtain
\begin{equation}\label{LinfMJEyJ}
    \liminf_{e\to\be}\frac{\mJ(\ou_{\be},e)-\mJ(\ou_{\be},\be)-\langle(\he^*_y,\he^*_J),(e_y,e_J)-(\be_y,\be_J)\rangle}
    {\|(e_y,e_J)-(\be_y,\be_J)\|_{L^2(\Omega)\times L^2(\Omega)}}\geq0.
\end{equation}
We have
\begin{equation}\label{ExpMJEyJ}
\begin{aligned}
    &\quad~\frac{\mJ(\ou_{\be},e)-\mJ(\ou_{\be},\be)-\langle(\he^*_y,\he^*_J),(e_y,e_J)-(\be_y,\be_J)\rangle}
                {\|(e_y,e_J)-(\be_y,\be_J)\|_{L^2(\Omega)\times L^2(\Omega)}}\\
    &=\frac{\mJ(\ou_{\be},e)-\mJ(\ou_{\be},\be)
     -\big\langle\big(\mJ'_{e_y}(\ou_{\be},\be),\mJ'_{e_J}(\ou_{\be},\be)\big),(e_y,e_J)-(\be_y,\be_J)\big\rangle}
                {\|(e_y,e_J)-(\be_y,\be_J)\|_{L^2(\Omega)\times L^2(\Omega)}}\\
    &\quad+\frac{\big\langle\big(\mJ'_{e_y}(\ou_{\be},\be)-\he^*_y,\mJ'_{e_J}(\ou_{\be},\be)-\he^*_J\big),
     (e_y,e_J)-(\be_y,\be_J)\big\rangle}{\|(e_y,e_J)-(\be_y,\be_J)\|_{L^2(\Omega)\times L^2(\Omega)}}.
\end{aligned}
\end{equation}
Since $\mJ(\cdot,\cdot)$ is Fr\'echet differentiable in $(e_y,e_J)$
at the point $(\be_y,\be_J)$, we have
$$\lim_{(e_y,e_J)\to(\be_y,\be_J)}\frac{\mJ(\ou_{\be},e)-\mJ(\ou_{\be},\be)-\big\langle\big(\mJ'_{e_y}(\ou_{\be},\be),
\mJ'_{e_J}(\ou_{\be},\be)\big),(e_y,e_J)-(\be_y,\be_J)\big\rangle}
{\|(e_y,e_J)-(\be_y,\be_J)\|_{L^2(\Omega)\times L^2(\Omega)}}=0.$$
Hence, from \eqref{LinfMJEyJ} and \eqref{ExpMJEyJ} it follows that
$$\liminf_{(e_y,e_J)\to(\be_y,\be_J)}\frac{\big\langle\big(\mJ'_{e_y}(\ou_{\be},\be)-\he^*_y,
\mJ'_{e_J}(\ou_{\be},\be)-\he^*_J\big),(e_y,e_J)-(\be_y,\be_J)\big\rangle}{\|(e_y,e_J)-(\be_y,\be_J)\|_{L^2(\Omega)
\times L^2(\Omega)}}\geq0,$$ which yields
$\big(\mJ'_{e_y}(\ou_{\be},\be)-\he^*_y,
\mJ'_{e_J}(\ou_{\be},\be)-\he^*_J\big)=(0_{L^2(\Omega)},0_{L^2(\Omega)})$.
Hence, we have
\begin{equation}\label{NcsCnd22}
\begin{cases}
    \he^*_y=\mJ'_{e_y}(\ou_{\be},\be)=\mJ'_u(\ou_{\be},\be)=\varphi_{\ou_{\be},\be},\\
    \he^*_J=\mJ'_{e_J}(\ou_{\be},\be)=G(\ou_{\be}+\be_y)=y_{\ou_{\be}+\be_y}.
\end{cases}
\end{equation}

Let
$e=(\be_y,\be_J,e_\alpha,e_\beta)\to(\be_y,\be_J,\be_\alpha,\be_\beta)=\be$
with $e_\alpha-\be_\alpha=e_\beta-\be_\beta$ and let $u\in\mG(e)$
with $u-\ou_{\be}=e_\alpha-\be_\alpha$. Note that $u\to\ou_{\be}$ as
$e\to\be$, and $\mJ(u,\be)=\mJ(u,e)\geq\mu(e)$. From
\eqref{DefSubMU} it follows that
\begin{equation}\label{LinfMJuAB}
    \liminf_{e_\alpha\to\be_\alpha}
    \frac{\mJ(u,\be)-\mJ(\ou_{\be},\be)-\langle\he^*_\alpha+\he^*_\beta,e_\alpha-\be_\alpha\rangle}
    {2\|e_\alpha-\be_\alpha\|_{L^\infty(\Omega)}}\geq0.
\end{equation}
By the choice of $e_\alpha$, $e_\beta$, and $u$ as above, we have
$$\begin{aligned}
    &\quad~\frac{\mJ(u,\be)-\mJ(\ou_{\be},\be)-\langle\he^*_\alpha+\he^*_\beta,e_\alpha-\be_\alpha\rangle}
    {2\|e_\alpha-\be_\alpha\|_{L^\infty(\Omega)}}\\
    &=\frac{\mJ(u,\be)-\mJ(\ou_{\be},\be)-\langle\mJ'_u(\ou_{\be},\be),u-\ou_{\be}\rangle
     +\langle\mJ'_u(\ou_{\be},\be)-\he^*_\alpha-\he^*_\beta,e_\alpha-\be_\alpha\rangle}
     {2\|e_\alpha-\be_\alpha\|_{L^\infty(\Omega)}}\\
    &=\frac{o\big(\|u-\ou_{\be}\|_{L^2(\Omega)}\big)
     +\langle\mJ'_u(\ou_{\be},\be)-\he^*_\alpha-\he^*_\beta,e_\alpha-\be_\alpha\rangle}
     {2\|e_\alpha-\be_\alpha\|_{L^\infty(\Omega)}}.
\end{aligned}$$
From this and \eqref{LinfMJuAB} with noting that
$$\lim_{e_\alpha\to\be_\alpha}
\frac{o\big(\|u-\ou_{\be}\|_{L^2(\Omega)}\big)}{2\|e_\alpha-\be_\alpha\|_{L^\infty(\Omega)}}=0$$
we deduce
\begin{equation*}
    \liminf_{e\to\be}
    \frac{\langle\mJ'_u(\ou_{\be},\be)-\he^*_\alpha-\he^*_\beta,e_\alpha-\be_\alpha\rangle}
    {2\|e_\alpha-\be_\alpha\|_{L^\infty(\Omega)}}\geq0,
\end{equation*}
which yields
$\mJ'_u(\ou_{\be},\be)-\he^*_\alpha-\he^*_\beta=0_{L^\infty(\Omega)^*}$.
Hence, we obtain
\begin{equation}\label{NcsCnd33}
    \he^*_\alpha+\he^*_\beta=\mJ'_u(\ou_{\be},\be)=\varphi_{\ou_{\be},\be}.
\end{equation}
Combining \eqref{NcsCnd11}, \eqref{NcsCnd22} and \eqref{NcsCnd33} we
obtain \eqref{DKFrSbDiff2}.

Conversely, we show that if
$\he^*=(\he^*_y,\he^*_J,\he^*_\alpha,\he^*_\beta)$ from $E^*_1$
holds \eqref{DKFrSbDiff2}, then $\he^*$ also holds \eqref{DefSubMU}
and thus $\he^*\in\widehat{\partial}\mu(\be)\cap E^*_1$. Note that
$$\mJ'_u(\ou_{\be},\be)=\varphi_{\ou_{\be},\be}=\he^*_\alpha+\he^*_\beta\quad\mbox{and}\quad\mJ'_e(\ou_{\be},\be)
=\big(\varphi_{\ou_{\be},\be},y_{\ou_{\be}+\be_y},0_{L^\infty(\Omega)^*},0_{L^\infty(\Omega)^*}\big).$$
Thus, by \eqref{DKFrSbDiff2} we have
\begin{equation}\label{ExpHEmJe}
    \he^*=\mJ'_e(\ou_{\be},\be)+\big(0_{L^2(\Omega)},0_{L^2(\Omega)},\he^*_\alpha,\he^*_\beta\big).
\end{equation}
Using the local upper H\"{o}lderian selection $h(\cdot)$ of the
solution map $S(\cdot)$ with $h(\be)=\ou_{\be}$ and $h(e)=\ou_e$ for
all $e\in\oB_\eta(\be)\cap\dom\mG$, from \eqref{ExpHEmJe} we deduce
that
$$\begin{aligned}
    \frac{\mu(e)-\mu(\be)-\langle\he^*,e-\be\rangle}{\|e-\be\|_E}
    &=\frac{\mJ(\ou_e,e)-\mJ(\ou_{\be},\be)-\langle\he^*,e-\be\rangle}{\|e-\be\|_E}\\
    &=\frac{\mJ(\ou_e,e)-\mJ(\ou_{\be},e)-\langle\he^*_\alpha,e_\alpha-\be_\alpha\rangle
      -\langle\he^*_\beta,e_\beta-\be_\beta\rangle}{\|e-\be\|_E}\\
    &\quad+\frac{\mJ(\ou_{\be},e)-\mJ(\ou_{\be},\be)
      -\langle\mJ'_e(\ou_{\be},\be),e-\be\rangle}{\|e-\be\|_E}\\
    &=\frac{\mJ(\ou_e,e)-\mJ(\ou_{\be},e)-\langle\mJ'_u(\ou_{\be},\be),\ou_e-\ou_{\be}\rangle}{\|e-\be\|_E}\\
    &\quad+\frac{\langle\he^*_\alpha+\he^*_\beta,\ou_e-\ou_{\be}\rangle-\langle\he^*_\alpha,e_\alpha-\be_\alpha\rangle
      -\langle\he^*_\beta,e_\beta-\be_\beta\rangle}{\|e-\be\|_E}\\
    &\quad+\frac{\mJ(\ou_{\be},e)-\mJ(\ou_{\be},\be)
      -\langle\mJ'_e(\ou_{\be},\be),e-\be\rangle}{\|e-\be\|_E},
\end{aligned}$$
and thus we have
$$\begin{aligned}
    \frac{\mu(e)-\mu(\be)-\langle\he^*,e-\be\rangle}{\|e-\be\|_E}
    &=\frac{\mJ(\ou_e,e)-\mJ(\ou_{\be},e)-\langle\mJ'_u(\ou_{\be},e),\ou_e-\ou_{\be}\rangle}{\|e-\be\|_E}\\
    &\quad+\frac{\langle\mJ'_u(\ou_{\be},e)-\mJ'_u(\ou_{\be},\be),\ou_e-\ou_{\be}\rangle}{\|e-\be\|_E}\\
    &\quad+\frac{\langle\he^*_\alpha+\he^*_\beta,\ou_e-\ou_{\be}\rangle-\langle\he^*_\alpha,e_\alpha-\be_\alpha\rangle
      -\langle\he^*_\beta,e_\beta-\be_\beta\rangle}{\|e-\be\|_E}\\
    &\quad+\frac{\mJ(\ou_{\be},e)-\mJ(\ou_{\be},\be)
      -\langle\mJ'_e(\ou_{\be},\be),e-\be\rangle}{\|e-\be\|_E}\\
    &=:\rho_1(e)+\rho_2(e)+\rho_3(e)+\rho_4(e).
\end{aligned}$$
By Lemma~\ref{lemma44}, we can find $\eta>0$ and $K_M>0$ such that
the following inequality
\begin{equation}\label{InqMJuvw}
    \big|\mJ''_u(u,e)(v_1,v_2)\big|\leq K_M\|z^e_{u,v_1}\|_{L^2(\Omega)}\|z^e_{u,v_2}\|_{L^2(\Omega)}
\end{equation}
holds for all $e\in B_\eta(\be)$, $u\in \mU_{ad}(e)$, and
$v_1,v_2\in L^2(\Omega)$. By \cite[Lemma~4.2]{QuiWch17}, we get
$$\|z^e_{u,v}\|_{L^2(\Omega)}=\|z_{u+e_y,v}\|_{L^2(\Omega)}\leq C_3\|v\|_{L^1(\Omega)},~\forall v\in L^1(\Omega),$$
for some constant $C_3>0$ independent of $u$ and $e$. From this and
\eqref{InqMJuvw} we infer that
\begin{equation}\label{InEqMJuvw}
    \big|\mJ''_u(u,e)(v_1,v_2)\big|\leq C\|v_1\|_{L^1(\Omega)}\|v_2\|_{L^1(\Omega)},~\forall v_1,v_2\in L^2(\Omega),
\end{equation}
where $C:=K_MC_3$. Using \eqref{InEqMJuvw} and \eqref{HldrSelct}, we
get
$$\begin{aligned}
  \lim_{e\to\be}|\rho_1(e)|
  &=\lim_{e\to\be}\frac{\big|\mJ(\ou_e,e)-\mJ(\ou_{\be},e)-\langle\mJ'_u(\ou_{\be},e),\ou_e-\ou_{\be}\rangle\big|}
                       {\|e-\be\|_E}\\
  &=\lim_{e\to\be}\frac{1}{2}\frac{\big|\mJ''_u(\hu_e,e)(\ou_e-\ou_{\be})^2\big|}{\|e-\be\|_E}
   \leq\lim_{e\to\be}\frac{C\|\ou_e-\ou_{\be}\|^2_{L^1(\Omega)}}{2\|e-\be\|_E}\\
  &=\lim_{e\to\be}\frac{C\|h(e)-h(\be)\|^2_{L^1(\Omega)}}{2\|e-\be\|_E}
   \leq\lim_{e\to\be}\frac{Cc^2\|e-\be\|_E^{2\ae}}{2\|e-\be\|_E}
   =0,
\end{aligned}$$
where $\hu_e=\ou_{\be}+\theta(\ou_e-\ou_{\be})$ for some function
$\theta(\cdot)$ with $0\leq\theta(x)\leq1$. In addition, by applying
Lemma~\ref{LemEsFDrMJ} we get for some constant $K_M>0$ that
$$\begin{aligned}
  \lim_{e\to\be}|\rho_2(e)|
  &=\lim_{e\to\be}
   \frac{\big|\langle\mJ'_u(\ou_{\be},e)-\mJ'_u(\ou_{\be},\be),\ou_e-\ou_{\be}\rangle\big|}{\|e-\be\|_E}\\
  &\leq\lim_{e\to\be}\frac{\|\mJ'_u(\ou_{\be},e)-\mJ'_u(\ou_{\be},\be)\|_{L^\infty(\Omega)}
   \|\ou_e-\ou_{\be}\|_{L^1(\Omega)}}{\|e-\be\|_E}\\
  &=\lim_{e\to\be}\frac{\|\varphi_{\ou_{\be},e}-\varphi_{\ou_{\be},\be}\|_{L^\infty(\Omega)}
   \|\ou_e-\ou_{\be}\|_{L^1(\Omega)}}{\|e-\be\|_E}\\
  &\leq\lim_{e\to\be}\frac{K_M\|e-\be\|_E\|\ou_e-\ou_{\be}\|_{L^1(\Omega)}}{\|e-\be\|_E}=0.
\end{aligned}$$
Moreover, by denoting $\Omega_i=\Omega_i(\be,\ou_{\be})$ for
$i=1,2,3$, it holds that
$\ou_{\be}|_{\Omega_1}=\alpha|_{\Omega_1}+\be_\alpha|_{\Omega_1}$
and $\ou_{\be}|_{\Omega_3}=\beta|_{\Omega_3}+\be_\beta|_{\Omega_3}$
due to the definition of $\Omega_1$ and $\Omega_3$. Hence, we obtain
$$\begin{aligned}
    \rho_3(e)
    &=\frac{\langle\he^*_\alpha+\he^*_\beta,\ou_e-\ou_{\be}\rangle
           -\langle\he^*_\alpha,e_\alpha-\be_\alpha\rangle-\langle\he^*_\beta,e_\beta-\be_\beta\rangle}{\|e-\be\|_E}\\
    &=\frac{\langle\he^*_\alpha|_{\Omega_1},\ou_e|_{\Omega_1}-\ou_{\be}|_{\Omega_1}-e_\alpha|_{\Omega_1}
           +\be_\alpha|_{\Omega_1}\rangle
           +\langle\he^*_\beta|_{\Omega_3},\ou_e|_{\Omega_3}-\ou_{\be}|_{\Omega_3}-e_\beta|_{\Omega_3}
           +\be_\beta|_{\Omega_3}\rangle}{\|e-\be\|_E}\\
  &=\frac{\langle\he^*_\alpha|_{\Omega_1},\ou_e|_{\Omega_1}-\alpha|_{\Omega_1}-e_\alpha|_{\Omega_1}\rangle
   +\langle\he^*_\beta|_{\Omega_3},\ou_e|_{\Omega_3}-\beta|_{\Omega_3}-e_\beta|_{\Omega_3}\rangle}{\|e-\be\|_E}\\
  &\geq0.
\end{aligned}$$
Finally, since the map $e\mapsto\mJ(u,e)$ is Fr\'echet
differentiable at $\be$ for each $u$, we get
$$\begin{aligned}
  \lim_{e\to\be}\rho_4(e)
  &=\lim_{e\to\be}\frac{\mJ(\ou_{\be},e)-\mJ(\ou_{\be},\be)-\langle\mJ'_e(\ou_{\be},\be),e-\be\rangle}{\|e-\be\|_E}=0.
\end{aligned}$$
Summarizing the above we deduce that
$$\liminf_{e\to\be}\frac{\mu(e)-\mu(\be)-\langle\he^*,e-\be\rangle}{\|e-\be\|_E}=\liminf_{e\to\be}\rho_3(e)\geq0.$$
This implies that $\he^*\in\widehat{\partial}\mu(\be)\cap E^*_1$.
$\hfill\Box$
\end{proof}

\medskip
From Theorem~\ref{ThmLUpHldr} and Theorem~\ref{ThmLwEsFrSb} we
obtain the following characterization of regular subgradients of the
marginal function $\mu(\cdot)$ in $E^*_1$.

\begin{Theorem}\label{ThmChRSbE1}
Let $(\be,\ou_{\be})\in\gph S$ and assume that all the assumptions
of Theorem~\ref{ThmLUpHldr} hold, where {\rm\textbf{(A4)}} holds
with $\ae>1/2$. Then, we have
$\he^*=(\he^*_y,\he^*_J,\he^*_\alpha,\he^*_\beta)\in\widehat{\partial}\mu(\be)\cap
E^*_1$ if and only if
$\he^*=(\he^*_y,\he^*_J,\he^*_\alpha,\he^*_\beta)\in E^*_1$
satisfies
\begin{equation}\label{ChrSbGrEst1}
\begin{cases}
   \he^*_y=\varphi_{\ou_{\be},\be},\\
   \he^*_J=y_{\ou_{\be}+\be_y},\\
   \he^*_\alpha|_{\Omega_1(\be,\ou_{\be})}\geq0,\,
   \he^*_\alpha|_{\Omega\setminus\Omega_1(\be,\ou_{\be})}=0,\\
   \he^*_\beta|_{\Omega_3(\be,\ou_{\be})}\leq0,\,\he^*_\beta|_{\Omega\setminus\Omega_3(\be,\ou_{\be})}=0,\\
   \he^*_\alpha+\he^*_\beta=\varphi_{\ou_{\be},\be}.
\end{cases}
\end{equation}
\end{Theorem}
\begin{proof}
Applying Theorems~\ref{ThmLUpHldr} and \ref{ThmLwEsFrSb}, we obtain
the assertion of the theorem. $\hfill\Box$
\end{proof}

\begin{Corollary}
Let $(\be,\ou_{\be})\in\gph S$ and assume that all the assumptions
of Corollary~\ref{CorLoUpLpSe} hold, where {\rm\textbf{(A4)}} holds
with $\ae>1/2$. Then, we have
$\he^*=(\he^*_y,\he^*_J,\he^*_\alpha,\he^*_\beta)\in\widehat{\partial}\mu(\be)\cap
E^*_1$ if and only if
$\he^*=(\he^*_y,\he^*_J,\he^*_\alpha,\he^*_\beta)\in E^*_1$
satisfies \eqref{ChrSbGrEst1}.
\end{Corollary}
\begin{proof}
It follows directly from Corollary~\ref{CorLoUpLpSe} and
Theorem~\ref{ThmLwEsFrSb}. $\hfill\Box$
\end{proof}

\begin{Remark}\rm
If all the assumptions of Theorem~\ref{ThmLwEsFrSb} hold around
$\be\in E$, then by using \eqref{LimSubdif} and \eqref{SngSubdif} we
obtain the following lower estimates for the Mordukhovich and
singular subdifferentials of $\mu(\cdot)$ via \eqref{DKFrSbDiff2} as
follows
\begin{equation}\label{LwEsMrSbDf}
    \partial\mu(\be)\supset\partial\mu(\be)\cap E^*_1\supset\Limsup_{e\stackrel{\mu}\longrightarrow\be}
    \Big(\widehat{\partial}\mu(e)\cap E^*_1\Big),
\end{equation}
and
\begin{equation}\label{LwEsSgSbDf}
    \partial^\infty\mu(\be)\supset\partial^\infty\mu(\be)\cap E^*_1\supset
    \Limsup_{\substack{e\stackrel{\mu}\longrightarrow\be\\ \lambda\downarrow0}}
    \lambda\Big(\widehat{\partial}\mu(e)\cap E^*_1\Big).
\end{equation}
If, in addition, there exists a sequence $e_n\to\be$ with
$\widehat{\partial}\mu(e_n)\cap E^*_1\neq\emptyset$, then
$0\in\partial^\infty\mu(\be)$. Indeed, since
$\widehat{\partial}\mu(e_n)\cap E^*_1$ is bounded by
\eqref{DKFrSbDiff2}. Combining this with \eqref{LwEsSgSbDf} yields
$0\in\partial^\infty\mu(\be)$.
\end{Remark}

\section{Concluding remarks}
\setcounter{equation}{0}

In this paper, we have obtained new results on differential
stability of a class of optimal control problems of semilinear
elliptic PDEs. We have established upper estimates for the regular,
the Mordukhovich, and the singular subdifferentials of the marginal
function $\mu(\cdot)$ in the setting that $E=L^2(\Omega)^4$ and
$\mQ\times\mU_{ad}(e)\subset L^2(\Omega)\times L^2(\Omega)$.
Furthermore, we have also obtained a new result on the existence of
local upper H\"{o}lderian selections of the solution map $S(\cdot)$
as well as a lower estimate for the regular subdifferential of
$\mu(\cdot)$ with respect to $E=L^2(\Omega)^2\times
L^\infty(\Omega)^2$ and $\mQ=L^{p_0}(\Omega)$, where $p_0>N/2$. In
the last setting, the problem of computing/estimating the limiting
subdifferentials of $\mu(\cdot)$ is very complicated since the
parametric space $E=L^2(\Omega)^2\times L^\infty(\Omega)^2$ is not
Asplund. For this reason, such problem remains open.

Further investigation, we are also interested in the problem of
computing subdifferentials of the marginal function $\mu(\cdot)$
with respect to the parametric space
\begin{equation}\label{PmtrSpcEOm}
    E=L^2(\Omega)\times L^2(\Omega)\times C(\bar\Omega)\times C(\bar\Omega).
\end{equation}
Note that for the parametric space $E$ given by \eqref{PmtrSpcEOm},
subgradients of the marginal function $\mu(\cdot)$ will be in the
form
$$e^*=(e^*_y,e^*_J,e^*_\alpha,e^*_\beta)\in E^*=L^2(\Omega)\times L^2(\Omega)\times C(\bar\Omega)^*\times
C(\bar\Omega)^*,$$ where $e^*_\alpha$ and $e^*_\beta$ are measures.
We think that the problem of characterization of subgradients
$e^*=(e^*_y,e^*_J,e^*_\alpha,e^*_\beta)\in E^*$ of $\mu(\cdot)$ with
functions $e^*_y$, $e^*_J$ and measures $e^*_\alpha$, $e^*_\beta$ is
an interesting and meaningful topic.

\end{document}